\documentclass[10pt]{amsart}
\usepackage[left=1in,right=1in,top=1in,bottom=1in]{geometry}
\usepackage[utf8]{inputenc}
\usepackage{setspace}
\usepackage{amsmath,amssymb}
\usepackage{comment}
\usepackage{xcolor}
\usepackage{url}
\usepackage[colorlinks=true,allcolors=blue,backref=page]{hyperref}
\usepackage[noabbrev,capitalize,nameinlink]{cleveref}
\usepackage{bbm}
\usepackage{centernot}
\usepackage[shortlabels]{enumitem}
\usepackage{extarrows}

\newtheorem{thm}{Theorem}

\newtheorem{lem}[thm]{Lemma}
\newtheorem{cor}[thm]{Corollary}
\newtheorem{prop}[thm]{Proposition}
\newtheorem{claim}[thm]{Claim}
\newtheorem{observation}[thm]{Observation}

\theoremstyle{remark}
\newtheorem{remark}[thm]{Remark}

\theoremstyle{definition}
\newtheorem{example}[thm]{Example}

\newcommand{\R}{\mathbb{R}}
\newcommand{\Z}{\mathbb{Z}}
\newcommand{\Prob}{\mathbb{P}}
\newcommand{\eps}{\varepsilon}

\renewcommand{\epsilon}{\eps}
\renewcommand{\P}{\mathbb{P}}

\newcommand{\E}{\mathbb{E}}
\newcommand{\ind}{\mathbbm{1}}

\newcommand{\lam}{\lambda}

\newcommand{\Ghat}{\widehat{G}}
\newcommand{\upperlambda}{\overline{\lambda}}
\newcommand{\lowerlambda}{\underline{\lambda}}
\newcommand{\mubar}{\overline{\mu}}

\newcommand{\cC}{\mathcal{C}}

\newcommand{\GH}{{\mathrm{GH}}}
\newcommand{\one}{\ind}

\title{The geometry of the giant component of random geometric graphs}
\author{Karoline Dubin}
\address{University of Illinois Chicago}
\email{kdubin3@uic.edu}

\author{Christian Gorski}
\address{University of Washington}
\email{cgorski1@uw.edu}

\author{Marcus Michelen}
\address{Northwestern University}
\email{michelen@northwestern.edu}

\date{}

\begin{document}

\begin{abstract}
    Consider a random geometric graph $G_M(n;r)$ whose vertex set consists of $n$ points chosen independently and uniformly from a Riemannian manifold $M$, with edges joining pairs of vertices whose distance in the metric $d_M$ is at most $r$.  Let $\Delta$ denote the expected average degree of the graph.  As is the case for Erd\H{o}s-R\'enyi graphs, there is a critical value $\Delta_c$, depending only on the dimension of $M$, such that if $\Delta > \Delta_c$ then $G_M(n;r)$ has a giant component.  We show that whenever $\Delta > \Delta_c$, the giant component of $G_M(n;r)$, equipped with the graph distance,  converges to the underlying manifold $M$ in the Gromov-Hausdorff distance after rescaling by an appropriate deterministic factor.  Our result holds for $\Delta$ depending on $n$ as well, provided $\Delta = o(n)$ and $\Delta \geq \Delta_c + \eps$ for any fixed $\eps > 0$. As a consequence, we show that for any pair of non-isometric compact Riemannian manifolds $M_1$ and $M_2$, there is a polynomial-time algorithm that distinguishes random geometric graphs on $M_1$ and $M_2$ throughout this regime of $\Delta.$   In the thermodynamic regime---i.e.\ when $\Delta$ is constant---our results appear to be new even in the classical cases where $M$ is a sphere or a torus.  Our proof makes use of techniques from first-passage percolation which allow us to understand the long-range behavior of the graph distance on small, approximately Euclidean patches of $M$,  together with global arguments that glue these local estimates into a global description.
\end{abstract}

\maketitle

\section{Introduction}

Let $M$ be a compact $d$-dimensional Riemannian manifold. For parameters $n$ and $r > 0$, define the \emph{random geometric graph} $G_M(n;r)$ to be the random graph whose vertex set consists of $n$ points $v_1,\ldots,v_n$ placed on $M$ uniformly at random and whose edge set consists of pairs $\{v_i,v_j\}$ with $d_M(v_i,v_j) \leq r$.  We are interested in the following general question: how can we understand the geometry of $M$ based on the graph $G_M(n;r)$?

It is useful to parameterize $r$ in terms of the average degree.  With this in mind, let $\Delta$ be the expected average degree of $G_M(n;r)$.  Classical results on random geometric graphs tell us that when $M$ is a box in $\R^d$ and $d \geq 2$, 
there is some constant $\Delta_c = \Delta_c(d) \in (1,\infty)$ depending only on $d$ so that if $\Delta > \Delta_c$, there is a unique giant component, meaning that there is only one component of the graph with $\Omega(n)$ many vertices \cite[Chapter~11]{Pen03}.  In stark contrast to this, when $\Delta < \Delta_c$, all components are only polylogarithmic in size with high probability.  This same phase transition occurs  when $M$ is a connected and compact Riemannian manifold of underlying dimension $d \geq 2$, which leads to a fundamental question: in the case of $\Delta > \Delta_c$, what can we deduce about the underlying manifold $M$ from the graph $G_M(n;r)$?  

Our main result shows that when $M$ is any compact $d$-dimensional Riemannian manifold, the random metric space given by the graph distance on the giant component of $G_M(n;r)$ is approximately isometric to the underlying metric space $(M,d_M)$.  To state this result, we will ultimately use the framework of the Gromov-Hausdorff distance, and so we let $d_\GH$ be the Gromov-Hausdorff distance.  The amount we will need to rescale by will depend on both the connection radius $r$ and a (continuous, deterministic, non-increasing) function $\mubar: (\Delta_c,\infty) \to (1,\infty)$ depending on the dimension; we will see this function is closely related to the time constant of a (continuum) first-passage percolation model.  We are ready to state our main theorem:

\begin{thm}\label{thm:degree}
    Let $M$ be a connected compact $d$-dimensional Riemannian manifold with $d \geq 2$ and let $\Delta_c = \Delta_c(d)$ be the critical average degree for having a giant component.  For each fixed $\Delta_0 > \Delta_c$ and $\eps > 0$, there exists $\alpha = \alpha(M,\Delta_0,\eps) > 0$, $c = c(M,\Delta_0,\eps) > 0$ and $n_0 = n_0(M,\Delta_0,\eps)$ so that the following holds.  Let $G = G_M(n;r)$ be the random geometric graph on $M$ with expected average degree $\Delta \in [\Delta_0,\alpha n].$  Let $\mathring{G}_M(n;r)$ be the  component of $G$ with largest diameter and $d_G$ be the graph metric on $G$.  
    Then $$\P\left(d_{\GH}\left(\left(\mathring{G}_M(n;r),  \frac{r}{\mubar(\Delta)} d_G\right),(M,  d_M) \right) > \eps\right) \leq \exp(- c \log^2 n)$$ for all $n \geq n_0$.
\end{thm}

The most well-studied examples of random geometric graphs are taken with $M$ either equal to a sphere or a (flat) torus; \cref{thm:degree} is new even in these cases in the ``thermodynamic regime'' when $\Delta_c < \Delta = O(1)$.   
We will prove \cref{thm:degree} in two cases depending on the degree.  The more challenging case will be $\Delta \in[\Delta_0,n^{1 - \beta}]$ for some $\beta > 0$, while the case of $\Delta > n^{1 - \beta}$ will be more straightforward.   

The critical values $\Delta_c$ will appear due to a connection with  continuum percolation: if one places a Poisson process of intensity $\lambda > 0$ on $\R^d$ and connects points that are within distance $1$ of each other, then for $d \geq 2$ there will be some  $\lambda_c(d)$ so that for $\lambda > \lambda_c(d)$ there is an infinite connected component almost-surely, while for $\lambda < \lambda_c(d)$ there almost-surely will not be.  If one lets $\Delta$ be the average degree of a vertex in this setting, one obtains a critical average degree $\Delta_c$ from this process, which coincides with our $\Delta_c$ in \cref{thm:degree}.   
While the critical values $\Delta_c$ are not known precisely in any dimension---except the trivial case of $\Delta_c = +\infty$ for $d = 1$---extensive simulations in the physics literature provide estimates; we include predictions\footnote{The critical parameter they present $\eta_c$ is related to our $\Delta_c$ via $\Delta_c = 2^d \eta_c$.} from Torquato and Jiao \cite{torquato2012effect}, where we round to three decimal places:

\begin{table}[ht]
\centering
\begin{tabular}{c|c}

$d$ & predicted $\Delta_c$ \\
\hline
2  & 4.512 \\
3  & 2.735 \\
4  & 2.086 \\
5  & 1.742 \\
6  & 1.497 \\
7  & 1.345 \\
8  & 1.255 \\
9  & 1.205 \\
10 & 1.165 \\
11 & 1.133 \\
\end{tabular}
\caption{Predicted critical average degree $\Delta_c$ for continuum percolation in $\mathbb{R}^d$.}
\label{tab:Delta_c}
\end{table}

While $\Delta_c(d)$ appears to be decreasing in $d$, it seems that this is not rigorously known.  However, a theorem of Penrose \cite{penrose1996continuum} shows that $\lim_{d \to \infty} \Delta_c(d) = 1$.  With this in mind, \cref{thm:degree} shows that in sufficiently high dimensions, one only needs the average degree to be slightly above $1$ in order to infer the underlying geometry of the manifold.

In both regimes of $\Delta < n^{1-\beta}$ and $\Delta > n^{1 - \beta}$, we will make use of the Poissonized version of the random geometric graph and it will be simple to deduce \cref{thm:degree} from its Poissonized versions.  
We will let $G_M^n(r)$ be the random graph on $M$ whose vertex set $\{v_j\}$ is the set of points of a Poisson process of intensity $n$ and whose edges consist of pairs $\{v_i,v_j\}$ with $d_M(v_i,v_j) \leq r$.    We now state the more challenging part, where $\Delta$ is allowed to be arbitrarily close to the critical average degree $\Delta_c$.

\begin{thm}\label{thm:main}
     Let $M$ be a connected compact $d$-dimensional Riemannian manifold with $d \geq 2$ and let $\Delta_c = \Delta_c(d)$ be the critical average degree for having a giant component.  For each fixed $\Delta_0 > \Delta_c$,  $\beta \in (0,1)$, $\eps > 0$, there is a $c = c(M,\Delta_0,\beta,\eps) > 0$ and $n_0 = n_0(M,\Delta_0,\beta,\eps)$ so that the following holds.  Let $G = G_M^n(r)$ be the Poisson random geometric graph on $M$ with expected average degree $\Delta \in [\Delta_0,n^{1-\beta}].$  Let $\mathring{G}_M^n(r)$ be the  component of $G$ with largest diameter and $d_G$ the graph metric on $G$. Then for all $n \geq n_0$ we have $$\P\left(d_{\GH}\left(\left(\mathring{G}_M^n(r),  \frac{r}{\mubar(\Delta)} d_G\right),(M,  d_M) \right) > \eps\right) \leq \exp(-c \log^2 n)\,.$$
\end{thm}

The general strategy of the proof will be to look at the graph in a small neighborhood of a given point $x \in M$.  Since $M$ is a manifold, a small neighborhood of $x$ will look approximately Euclidean.  This implies that the graph $G_M^n(r)$ near $x$ will look approximately like the random geometric graph in $\R^d$ whose vertex set is a (nearly homogeneous) Poisson process and whose edge set is those whose distances are (approximately) at most $r$.  This is a more classical setting of continuum percolation, and we refer the reader to the books \cite{Pen03,MR1409145} for many statements in this setting.  In the literature of percolation, graph distances on supercritical percolation clusters are often referred to as the \emph{chemical distance} between two points.  Tools to analyze the chemical distance come from the literature on \emph{first-passage percolation}, which most classically considers distances in $\Z^d$ where each edge is given an independent weight.   We will see from a quantitative shape theorem of Yao-Chen-Guo \cite{YCG11} that chemical distances in Euclidean continuum percolation are approximately equal to a rescaling of the Euclidean distance.  This will help us dictate the behavior of graph distances of $G_M^n(r)$ on small patches of $M$ and we will need to glue together a finite collection of these patches in order to understand the global behavior captured by \cref{thm:main}.  
We provide a more detailed proof sketch in \cref{sec:proof-sketch}.

The case of $\Delta > n^{1-\beta}$ is quite a bit simpler.  In particular, all the percolation arguments are no longer necessary.  We in fact have the function $\mubar$ is continuous (see \cref{Continuity of Time Constant}), non-increasing and $\mubar(\Delta) \downarrow 1$ as $\Delta \to \infty$ (see \cref{lem:limit-time-constant}).  In the case of polynomially large $\Delta$ we will also have that $G_M^n(r)$ is connected (\cref{lem:connected}) and so we need only refer to $G_M^n(r)$ rather than $\mathring{G}_M^n(r).$

\begin{thm}\label{thm:sparse}
    Let $M$ be a connected compact $d$-dimensional Riemannian manifold with $d \geq 2$.  For each fixed $\eps > 0$ and $\beta \in (0,1)$ there are $\alpha = \alpha(M,\beta,\eps) > 0, c = c(M,\beta,\eps) > 0$ and $n_0 = n_0(M,\beta,\eps)$ so that the following holds.  
    For $n \geq n_0$ and $\Delta \in [n^{1-\beta},\alpha n]$, let $G_M^n(r)$ be the  Poisson random geometric graph on $M$ of expected average degree $\Delta$.  Then $$\P(d_{\GH}\left(({G}_M^n(r),  r d_G),(M,  d_M) \right) > \eps) \leq \exp(- c n^{c})\,.$$
\end{thm}

We first show how \cref{thm:degree} follows immediately from \cref{thm:main} and \cref{thm:sparse}.  We quickly note that due to the continuity of $\mubar$ (\cref{Continuity of Time Constant}) along with concentration for the number of edges in the random geometric graph (\cref{lem:edge-concentration}) we may pass between the Poisson and fixed $n$ model without much issue.

\begin{proof}[Proof of \cref{thm:degree}]
    Let $r$ be so that the expected average degree of $G_M(n;r)$ is equal to $\Delta$.  Set $G_M^n(r)$ to be the Poisson random geometric graph with the same connectivity radius.  For $\alpha$ small enough, we have that the expected average degree of $G_M^n(r)$ is $\Delta' = (1 + o_{\alpha \to 0}(1))\Delta$ by \cref{lem:almostflat} along with \cref{lem:edge-concentration}.
    Let $V(G_M^n)$ denote the vertex set of $G_M^n$.  Set $\mathcal{G} = \{|V(G_M^n)| = n\}\,.$  Since $|V(G_M^n)|$ is a Poisson variable of mean $n$, Stirling's formula shows $\P(\mathcal{G}) = (1 + o(1)) (2\pi n)^{-1/2}$.  
    
    Note that for $\Delta$ sufficiently large we may use \cref{lem:limit-time-constant} to show $\mubar(\Delta') \in[1,1 + \eps/2]$, which will allow us to omit $\mubar(\Delta')$ in this regime by only replacing $\eps$ with $\eps/2$. 
    
    Define the event $$\mathcal{A} = \left\{d_{\GH}\left(\left(\mathring{G}_M^n(r),  \frac{r}{\mubar(\Delta')} d_G\right),(M,  d_M) \right) > \eps\right\} $$
     and note that \begin{align*}
     \P\left(d_{\GH}\left(\left(\mathring{G}_M(n;r),  \frac{r}{\mubar(\Delta')} d_G\right),(M,  d_M) \right) > \eps\right) = \P(\mathcal{A} \,|\,\mathcal{G}) \leq 3 \sqrt{n} \P(\mathcal{A}) \leq e^{-\Omega(\log^2 n)}
 \end{align*}
 where in the last inequality we used \cref{thm:main} and \cref{thm:sparse}.  Using \cref{Continuity of Time Constant} and recalling that $\Delta' = (1 + o_{\alpha \to 0}(1))\Delta$ completes the proof.
\end{proof}

\subsection{Algorithmic implications}  There has recently been much attention on statistical and algorithmic problems related to the geometry of random geometric graphs.  A natural question is when can one detect the difference between random geometric graphs and the Erd\H{o}s-R\'enyi random graph.  An influential work of Bubeck, Ding, Eldan, and R{\'a}cz \cite{BDJR16}  considers the case of constant edge density $p = \Omega(1)$ and  the regime where the dimension $d$ and the number of points $n$ simultaneously grow.  They show that for $d \gg n^3$ a random geometric graph of constant edge density is indistinguishable from the Erd\H{o}s-R\'enyi graph while for $d \ll n^3$ the graphs can be distinguished by counting signed triangles. Further work in this direction for the case of sparse $p$ was considered by Liu, Mohanty, Schramm and Yang \cite{LSSY22}.  In both cases, the underlying manifold of the geometric random graph is the sphere $\mathbb{S}^{d-1}$.

A question closer to our own is whether one can detect the underlying geometry itself.   An influential work by Bernstein,  De Silva, Langford, and Tenenbaum \cite{bernstein2000graph} from 2000 from the literature of \emph{manifold learning} analyzes a closely related problem; they seek understand an algorithm ``Isomap''  which has a crucial stage where one approximates the graph distance in a random geometric graph on a $d$-dimension manifold $M$ embedded in a higher dimension Euclidean stage with the intrinsic geodesic distance.  At the core of \cite{bernstein2000graph} is an estimate of distances that is analogous to case of large $\Delta$, i.e.\ \cref{thm:sparse}.  There has been much work on manifold learning since the work of \cite{bernstein2000graph}, although none appears to hold in the regime where the graph fails to be connected, e.g.\ when $\Delta = \Theta_d(1).$  

We highlight a  recent sequence of works by Huang, Jiradilok, and Mossel \cite{HJM24,HJM25,HJM26}  seeks to reconstruct the underlying geometry of noisy versions of random geometric graphs in an algorithmic fashion.  Throughout their series of papers, rather than simply adding edges deterministically for those of distance at most $r$, they add edges between points $x,y$ with some probability $p(d(x,y))$ for some function $p$.  This introduces a challenge to their problem that does not appear in our setting: the presence of additional edges---for instance edges that can be long in the $d_M$ metric---leads to a noisier and denser setting.  In particular, the graphs they consider need \emph{not} be approximately isometric to the underlying manifold, and one of their central challenges is to infer the ``signal'' of the underlying geometry in spite of this ``noise.''

The first work in their series \cite{HJM24} considers an embedded manifold $M \subset \R^d$ and  gives a $O(n^3)$ time algorithm for reconstructing the underlying metric in the dense regime (i.e.\ $r$ is a constant, or equivalently when $\Delta = \Theta(n)$). Their next work \cite{HJM25} extends their work to smaller average degree, in particular down to $\Delta \gg n^{1/2}\mathrm{polylog}(n)$.  Their very recent work \cite{HJM26} introduces a different technique for denoising and improves their results in higher dimensions.  We refer the reader to their papers and the references therein for more context on manifold learning.

As a consequence of our techniques, we will show that one can detect geometry in polynomial time for all average degree bounded away from the critical threshold $\Delta_c(d)$.  Since the distribution of the random geometric graph of average degree $\Delta$ is invariant under rescaling the manifold, we will normalize so our manifolds have volume $1$.
We  phrase our algorithmic result in the general setting of distinguishing between two manifolds that are not isometric:

\begin{thm}\label{thm:algorithmic}
	Let $M_1$ and $M_2$ be connected compact Riemannian manifolds of dimension $d_1, d_2 \geq 2$ and volume $1$ so that $M_1$ and $M_2$ are not isometric.  Let $\Delta_c(d_j)$ denote the critical average degree for having a giant component.  There is some $K = K(M_1,M_2) < \infty$ so that there is a deterministic algorithm depending on $M_1$ and $M_2$ with run time at most $n^K$ with the following properties.  There is $\alpha = \alpha(M_1,M_2)$ so that for each $\Delta_0 > \max\{\Delta_c(d_1),\Delta_c(d_2)\}$ there is $c = c(M_1,M_2,\Delta_0)$ so for all $\Delta \in [\Delta_0, \alpha n]$, given the adjacency matrices of the random geometric graphs  $G_{M_1}(n;r_1)$ and $G_{M_2}(n;r_2)$ of expected average degree $\Delta$, the algorithm can correctly identify which came from $M_1$ and $M_2$ with probability at least $1- \exp(-c \log^2 n)\,.$ 
\end{thm}

In the case of constant $\Delta > \Delta_c(d)$ this appears to be new even in the case of distinguishing a sphere and a torus.  The main idea is that if two compact manifolds are not isometric, then there is some finite collection of points in one so that the array of pairwise distances between those points is never achieved in the other (see \cref{lem:GH-positive-implies-finite}) and subsequently applying \cref{thm:degree}.  

\begin{remark}
    If $M_1$ and $M_2$ have different \emph{diameter}, then the runtime in \cref{thm:algorithmic} can be simplified.  In particular, the proof shows that we need only find the diameter of the graphs, and with probability at least $1 - \exp(-c \log^2 n)$ the graph coming from the manifold of larger diameter will have larger graph diameter.  One can create the array of pairwise distances by running a breadth-first search starting from each of the $n$ vertices in each graph.  Since a breadth-first search has run-time $|E|$ where $|E| = \Theta(n\Delta)$ is the number of edges in each graph, this implies that we can distinguish between manifolds of different diameter in time $O(n^2 \Delta)$ in the setting of \cref{thm:algorithmic}.
\end{remark}

\begin{example}[Sphere versus torus]
    If $M_1$ and $M_2$ have the same volume but different diameter then one example of a distinguishing statistic is the diameter of the giant components.  In particular, in the case of a sphere versus a flat torus, we can let $M_1 = \alpha_d \mathbb{S}^{d}$ where $\alpha_d$ is the normalization so that $M_1$ has volume $1$ and $M_2 = \R^d / \Z^d$, we note that $\mathrm{diam}(M_1) = \pi \alpha_d$ and $\mathrm{diam}(M_2) = \frac{\sqrt{d}}{2}$.  Since we have $$\alpha_d = \left( \frac{\Gamma(\frac{d+1}{2})}{2 \pi^{(d+1)/2}}\right)^{1/d}$$
    one can check that $\pi \alpha_d = \frac{\sqrt{d}}{2}$ if and only if $d = 1$.  In particular, for all $d \geq 2$ we have that $\mathrm{diam}(M_1) > \mathrm{diam}(M_2)$  and so choosing the graph of larger graph metric diameter distinguishes a sphere from a torus with high probability.
\end{example}

We next give an example where the manifolds have equal volume and diameter but are not isometric.

\begin{example}
    It is possible for $M_1$ and $M_2$ to have the same volume and diameter and still fail to be isometric.  For instance, if we consider $M_1 = \alpha_2 \mathbb{S}^2$ as in the previous example and now define $M_2 = \R^2 / (a \Z \oplus b \Z)$ with $ab = 1$, then $\mathrm{Vol}(M_2) = 1$.  We have that $$\mathrm{diam}(M_2) = \frac{1}{2}\sqrt{a^2 + b^2}\,.$$
    Since the case of $a = b = 1$ yields $\mathrm{diam}(M_2) < \mathrm{diam}(M_1)$, by the intermediate value theorem we may pick $a$ and $b$ so that we simultaneously have $ab = 1$ and $\mathrm{diam}(M_2) = \mathrm{diam}(M_1)$.  However, the two are not isometric since $M_1$ is simply connected and $M_2$ is not.
\end{example}

\subsection{Proof outline and structure of the paper} \label{sec:proof-sketch}

Before describing the giant component of random geometric graphs, we recall some intuition for the emergence of the giant component in the Erd\H{o}s-R\'enyi random graph $G(n,p)$.  Recall that if $G = G(n,p)$ for $p \in [0,1]$ then $G$ has $n$ vertices and each possible edge is added independently with probability $p$.  If we parameterize $p = \Delta / n$ where we think of $\Delta$ as fixed, then the average degree of $G$ is approximately $\Delta$.  Further, if we pick a random vertex $v \in G$, then it will have roughly $\mathrm{Poisson}(\Delta)$ many neighbors, each of which will have approximately $\mathrm{Poisson}(\Delta)$ neighbors and so on.  When $\Delta$ is fixed and $n$ tends to infinity, the depth $k$ ball around a random vertex $v$ will be indistinguishable from a branching process with offspring distribution $\mathrm{Poisson}(\Delta)$.  In the language of random graphs, this says that the \emph{local weak limit} of $G(n,\Delta/n)$ is the branching process with offspring distribution $\mathrm{Poisson}(\Delta)$.  Synonymously, one says that $G(n,\Delta/n)$ Benjamini-Schramm converges to the branching process with offspring distribution $\mathrm{Poisson}(\Delta)$.

The classical theory of branching processes tells us that when $\E \mathrm{Poisson}(\Delta) = \Delta \leq 1$ we have that the branching process is finite almost surely, while if  $\E \mathrm{Poisson}(\Delta) = \Delta > 1$ there is a positive probability that it is infinite.  When thinking about $G(n,\Delta/n)$ for $\Delta \leq 1$, this means that for a typical vertex its connected component has $O(1)$ many vertices.  This immediately implies that the largest component has $o(n)$ many vertices; more refined work shows that the largest component is size $O(\log n)$ for $\Delta < 1$ and $O(n^{1/3})$ for $\Delta = 1$. Conversely, for $\Delta > 1$, there is a positive probability that the underlying branching process is infinite.  Since $G(n,\Delta/n)$ has only finitely many nodes, this means that two nodes whose local limits are quite large must eventually meet.  This shows that for $\Delta > 1$ the largest connected component has size $\Omega(n).$  We refer the reader to \cite{bollobas2011random} for more context on $G(n,p)$ and this phase transition.

One can understand the case of random geometric graphs in analogy with the Erd\H{o}s-R\'enyi case by similarly thinking of Benjamini-Schramm convergence.   If we set $M$ to be a connected, compact Riemannian manifold of dimension $d$, then a small neighborhood of each point of $M$ is locally Euclidean.  Further, the volume form near a point will also be approximately Lebesgue when rescaled.  More formally, for a point $x \in M$, if we take $\delta > 0$ small enough, then there is a map $u_x: B_M(x,\delta) \to \R^d$ that maps the ball of radius $\delta$ centered at $x$ in $M$ that approximately preserves the measure and metric on $M$ (see \cref{lem:almostflat} for a formal statement).  The image of the component of $G_M(n;r)$ in $B_M(x,\delta)$ will have vertex set given by an (approximately) homogeneous Poisson process of intensity $n$ in $\R^d$ and whose edges consist of points within Euclidean distance (approximately) at most $r$.  In the case where $\Delta$ is fixed, we can rescale $\R^d$ so that points are connected if they are within distance $1$ and the intensity $\lambda$ is so that the average degree is $\Delta$ (one may use \cref{Palm theory} to show that $\lambda = \Delta / \mathrm{Vol}(B_{\R^d}(1))$).  This limiting process is a model of \emph{continuum percolation} that goes by various names, such as the Poisson-Boolean model or the Gilbert model, after being introduced by Gilbert in 1961 \cite{gilbert1961random}.   

In analogy with Erd\H{o}s-R\'enyi graphs, the relevant question is whether this continuum percolation model has an infinite component.  Classical results show that for $d \geq 2$ there is indeed a non-trivial phase transition, and there is some $\lambda_c = \lambda_c(d)$ so that if $\lambda < \lambda_c(d)$ there is almost-surely no infinite component while if $\lambda > \lambda_c(d)$ then there is almost-surely a unique infinite component.  We refer the reader to the classic book of Meester and Roy \cite{MR1409145} on continuum percolation for more history and theory related to this transition.  

This argument shows that the local weak limit of $G_M(n;r)$ is given by continuum percolation; this can be turned into an argument that shows a giant component phase transition occurs precisely as the average degree $\Delta$ crosses the critical value for continuum percolation in $\R^d$ where $d$ is the dimension of $M$.  Letting $\mathring{G}_M(n;r)$ denote the giant component of $G_M(n;r)$, this local weak limit in fact captures the \emph{geometry} of the giant component.  Graph geodesics in $\mathring{G}_M(n;r)$ between nearby points translate to graph geodesics in some approximate local weak limit, i.e.\ they are fairly close to graph geodesics in the infinite continuum percolation cluster in $\R^d$.  The graph distance in a supercritical percolation cluster is often referred to as the \emph{chemical distance}, and is typically analyzed using tools from \emph{first-passage percolation}.

The most classical setting of {first-passage percolation} consists of the underlying graph being the infinite lattice $\Z^d$ for $d\geq 2$ with i.i.d.\ non-negative random variables $\tau_e$ assigned to each edge referred to as its weight.  The weight of a path is given by the sum of the weights along the path, and the distance $T(x,y)$ between two points $x,y \in \Z^d$ is the smallest weight achievable by a path between the points.  Among the most classical statements of first-passage percolation is that for far away points, the random distance $T(x,y)$ is roughly deterministic: for each unit vector $\widehat{v}$ and edge distribution $\tau$ there is a constant $\mu( \widehat{v})$ known as the \emph{time constant} so that if $x - y$ is roughly parallel to $\widehat{v}$, then $T(x,y) \approx \mu(\widehat{v})\|x - y\|_2$.  A stronger version of this that is uniform in the direction $\widehat{v}$ is referred to as the \emph{shape theorem} and was proven by Cox and Durrett \cite{cox1981some} using a subadditive ergodic theorem by Kingman \cite{kingman1968ergodic} which immediately proves the case for fixed $\widehat{v}$.  We refer the reader to \cite{auffinger201750} for more context on first-passage percolation.

In our setting, the chemical distances on the infinite percolation cluster in $\R^d$ also may be seen to follow a shape theorem.  We will use a quantitative version of this by Yao-Chen-Guo \cite{YCG11} which we reproduce as \cref{large-deviations-time-constant}.  This will show that not only do small patches of $\mathring{G}_M(n;\Delta)$ look like the infinite cluster of continuum percolation from the perspective of connectivity, but also from the perspective of geometry.  Namely, in each small patch of $M$, we will show using \cref{large-deviations-time-constant} that distances within that patch are approximately  the time constant multiplied by the (rescaled) Euclidean distance.

In \cref{sec:move-giant-component}---which is the most technical section of the paper---we show that one can in fact move from the giant component of the finite random geometric graph $G_M(n;r)$ to the infinite component of continuum percolation that is its local weak limit.  From there, we apply the quantitative shape theorem of Yao-Chen-Guo in \cref{sec:patch} to show that on patches of $M$, we have that graph distances in $\mathring{G}_M(n;r)$ behave in a predictable fashion.  We then need to boost these patch estimates into global estimates, and do so by applying our results from \cref{sec:move-giant-component,sec:patch} to show that not only does a finite collection of patches capture much of the geometry involved, but that graph geodesics do not wander far enough to escape any of these patches.  \cref{sec:global-bounds} is split into two sections, the first of which \cref{sec:two-point} performs this gluing of patches to show that two-point distances in $\mathring{G}_M(n;r)$ are close to the rescaled version of $d_M$ while \cref{sec:proof-of-main} uses this estimate to build approximate embeddings that witness a bound on the Gromov-Hausdorff distance.  These three sections complete the proof of \cref{thm:main}.

The proof of the denser case \cref{thm:sparse} is simpler and requires no percolation input.  The proof proceeds in \cref{sec:denser} by showing an analogous two-point estimate by simply showing that in small patches along a $d_M$ geodesic in $M$ there are enough points so that the graph distance is roughly close to its minimum possible value.  We end with \cref{sec:algorithmic-proofs} where we deduce \cref{thm:algorithmic} from \cref{thm:degree}.  The main work there is to prove an elementary and likely standard fact about compact metric spaces \cref{lem:GH-positive-implies-finite} which shows that if two compact metric spaces are not isometric then there is a finite collection of points in one whose pairwise distances are never achieved in the other.  

To begin with, \cref{sec:preliminaries} consists of relevant preliminaries on Riemannian manifolds, Poisson processes, and first-passage percolation.  Some of these facts are more standard---such as the Palm theory for Poisson processes---while some are tailored to our setting, such as the continuity of the time constant in continuum percolation (\cref{Continuity of Time Constant}, recently proven by the first two authors).  We continue with a bookkeeping section, \cref{sec:setting-up-processes}, where we specify choices of parameters and trap our approximately Euclidean patch between two copies of $\R^d$ where we have only slightly rescaled the metric and measure.

\section{Preliminaries} \label{sec:preliminaries}

Throughout, $G_M^n(r)$ will be the geometric random graph on $M$ whose vertex set is given by a Poisson process of intensity $n$ and we connect vertices $v$ and $w$ if $d_M(v,w) \leq r$.  We will see below (see \cref{lem:almostflat}) that we will often map from a neighborhood of $M$ to Euclidean space.  There we will rescale space by $(1 + o(1))r$ and so we will define $\lambda := n r^d$.  While in the introduction we parameterize based on the average degree $\Delta$, we will parameterize by $\lambda$ instead since it will be more convenient when comparing to Euclidean space.   The relationship between $\Delta$ and $\lambda$ will be further described in \cref{sec:PPP-FPP}.  

We first describe a basic tool to compare our Riemannian manifold $M$ with $\R^d$ in \cref{sec:riemannian}, then describe Poisson point processes and first-passage percolation in \cref{sec:PPP-FPP}, and finally discuss the uniqueness of the giant component of supercritical random geometric graphs in \cref{sec:uniqueness}.

\subsection{Riemannian manifolds} \label{sec:riemannian}

Throughout, $M$ is a connected compact Riemannian manifold  of volume $1$.  We write $d_M$ for the metric induced by the Riemannian metric on $M$ and $dM$ for the Riemannian volume measure.  Central to our approach will be to compare a patch of $M$ with a set in Euclidean space.  We write $\mathcal{L}$ for Lebesgue measure in $\R^d$.  The following basic lemma will be our main tool:

\begin{lem}[Microscopic balls on $M$ are almost flat] \label{lem:almostflat}
    For each $z \in M$ and $\eta > 0$ there exists $\delta_0 \in (0,1)$
	such that for every $0<\delta \le \delta_0$,
	any ball $U = B_M(z,\delta)$ in $M$ of radius $\delta$ has the following properties. 
	Let
	\[
	u : U \to B_{\R^d}(0,\delta)
	\]
	be exponential normal coordinates. 
	Then $u$
	is a diffeomorphism with
	\[
	e^{-\eta} d_M(x,y) 
	\le d_{\R^d}( u(x), u(y) )
	\le e^{\eta}d_M(x,y)
	\]
	for all $x,y \in U$.
	Moreover if $\mathcal{L}$ is the Lebesgue measure on $\R^d$ then the Radon-Nikodym derivative of the pushforward $u_\ast dM$ of the Riemannian volume measure $dM$ satisfies $$e^{-\eta} \leq \frac{u_\ast dM}{d\mathcal{L}} \leq e^{\eta}$$ on $u(U)$.
\end{lem}
\begin{proof}
        Set $\phi = \log_z = (\exp_z)^{-1}: U \to T_z M \to \R^d$ where we use an orthonormal basis for $T_zM$ to
	identify $T_zM$ with $\R^d$, and 
	$\log_z$ is the inverse of 
	the exponential map $\exp_z:T_zM \to M$
	restricted to a small enough neighborhood
	of $0$ that it is a diffeomorphism.  For a point $y \in U$ we set $u = \phi(y)$ and note that in these coordinates the metric satisfies \begin{equation} \label{eq:normal-coords-error}
		g_{ij}(u) = \delta_{ij} + O(|u|^2)
	\end{equation} (see, e.g., \cite[Proposition 5.24]{lee2018introduction}).  Taking $\delta$ sufficiently small ensures that $\phi$ is $e^\eta$-bi-Lipschitz.  The Riemannian volume element $dM$ is \begin{equation}
		dM = \sqrt{\det(g_{ij})} du^1 \wedge  \cdots \wedge du^d
	\end{equation}
	(see \cite[Proposition 2.41]{lee2018introduction}) and so combining with \eqref{eq:normal-coords-error} shows $dM = (1 + O(|u|^2)) du^1 \wedge \cdots \wedge du^d$.  Taking $\delta$ small enough forces the $(1 + O(|u|^2))$ coefficient to lie in $[e^{-\eta},e^\eta]$ which then completes the proof.
\end{proof}

\subsection{Poisson processes and first-passage percolation} \label{sec:PPP-FPP}

Throughout we will make use of various aspects of Poisson processes.  A useful tool is the Palm theory for Poisson processes, certain versions of which are often called Mecke formulas; this will allow us to compute the expectation of various statistics of Poisson processes in terms of deterministic integrals.  We will often make use of these for estimating certain quantities and so we present a general version of this property:

\begin{thm}[Theorem 4.1 \cite{last2018lectures}]\label{Palm theory} \emph{[Palm theory for Poisson processes.]}
Let $\nu$ be a $\sigma$-finite measure on $\R^d$ and $X_\nu$ a Poisson process with intensity $\nu$.  Then for all non-negative, bounded and measurable functions $f$ we have $$\E\left[\sum_{x \in X_\nu} f(x,X_\nu) \right] = \int_{\R^d} \E\left[f(x,X_\nu \cup \{x\}) \right]\,d\nu(x)\,.$$ 
\end{thm}

We note that iterating \cref{Palm theory} yields a multivariate version, also known as the multivariate Mecke equation or the factorial Mecke equation (see \cite[Theorem 4.4]{last2018lectures} for a precise statement).

For any $r > 0$ we let $G_{\R^d}^\lambda(r)$ be the random geometric graph on $\R^d$ where the vertex set is a Poisson process of intensity $\lambda$ and points are connected by an edge if they are within distance at most $r$.  Note that for any $r > 0$ we have that the pushforward
of $G_{\R^d}^{\lambda}(r)$ 
under the map
$\R^d \to \R^d, x \mapsto \frac{1}{r}x$
has distribution equal to that of $G_{\R^d}^{\lambda/r^d}(1)$, and so we can always rescale Euclidean space so the connection radius is $r = 1$.

For each $d \geq 2$, there exists $\lambda_c(d) < \infty$ such that 
\[
   \Prob( G_{\R^d}^{\lambda}(1) \mbox{ has an infinite connected component} ) = 
   \begin{cases}
      0 & \lambda < \lambda_c(d) \\
      1 & \lambda > \lambda_c(d)
   \end{cases}
\]
and in fact there is always at most one infinite connected component \cite[Theorem~9.19]{Pen03}.
We write $C_\infty(\lambda)$ for the almost-surely unique infinite connected component of $G_{\R^d}^\lambda(1)$ for $\lambda > \lambda_c(d)$.  

By \cref{Palm theory}, the average degree $\Delta$ of $G_{\R^d}^\lambda(1)$ satisfies  $\Delta = \lambda \alpha_d$ where $\alpha_d := \mathcal{L}(B_{\R^d}(0,1))$ is the volume of the unit ball in $\R^d$.  The critical average degree $\Delta_c(d)$ can thus be compared to $\lambda_c(d)$ via $\Delta_c(d) = \alpha_d\lambda_c(d).$ 

For $\lambda > \lambda_c(d)$ and $x \in \R^d$, we may write $\mathring{x} = \mathring{x}(\lambda) \in C_\infty(\lambda)$ to be the closest point in $C_\infty(\lambda)$ to $x$.  Writing $d_\lambda$ for the graph distance, we are interested in understanding the behavior of $d_\lambda(\mathring{x},\mathring{y})$ for distant points $x,y \in \R^d$.  A quantitative version of the shape theorem will say that these large distances behave in deterministic ways.  We require a bit more notation to set this up. For events concerning $G_{\R^d}^\lambda$, we let $\P_{x,y}$ denote the probability measure conditioned on $x$ and $y$ lying in the vertex set.  For a graph $G$ and vertices $x,y \in V(G)$, write $x \xleftrightarrow{G} y$ for the event that $x$ and $y$ are connected in the graph.   The deterministic function $\mu:(\lambda_c,\infty) \to (1,\infty)$ is referred to as the \emph{time constant} and is dependent only on the dimension; we may view the following as a definition of $\mu(\lambda)$:

    \begin{thm}[Yao-Chen-Guo, \cite{YCG11}] \label{large-deviations-time-constant}
        Fix $\lam > \lam_c(d)$.   Let $G^\lambda = G^{\lam}_{\R^d}(1)$ and let  $d_\lambda$ denote the graph distance
        on $G^{\lam}$. Then, for any $\xi > 0$ there exists 
        $c(\lam,\xi,d)>0$ such that for all $x,y \in \R^d$ we have
        \[
            \Prob_{x,y}\left( x \xleftrightarrow{G^{\lam}} y,
            \frac{d_\lam(x,y)}{\mu(\lam)\|x - y\|}
            \notin [e^{-\xi}, e^\xi] 
            \right)
            \le \exp(-c\|x - y\|).
        \]
    \end{thm}

    We note that $c$ may depend on $\lambda$, and so we will see that we will only apply \cref{large-deviations-time-constant} for a finite collection of $\lambda$ which we will specify in \cref{sec:setting-up-processes}.  Using the relationship $\Delta = \alpha_d \lambda$, we may write $\mubar(\Delta) := \mu(\Delta/\alpha_d)$.
    
Since in $G_{\R^d}^\lambda(1)$ we always have $d_\lambda(x,y) \geq \|x - y\|$, we immediately have that $\mu(\lambda) \geq 1$ for all $\lambda > \lambda_c$.  A quick consequence of \cref{large-deviations-time-constant} will be that $\mu(\lambda)$ is non-increasing.  We need one more simple property, namely that $\mu(\lambda) \downarrow 1$ as $\lambda \to \infty$.

\begin{lem}\label{lem:limit-time-constant}
    The time constant $\mu(\lambda)$ is non-increasing with $\lim_{\lambda \to \infty} \mu(\lambda) = 1\,.$
\end{lem}
\begin{proof}
    We note that the fact that $\mu(\lambda)$ is non-increasing is a consequence of \cref{large-deviations-time-constant}.  To see that $\lim_{\lambda \to \infty} \mu(\lambda) = 1$, let $C_\infty(\lambda)$ denote the unique infinite cluster of $G_{\R^d}^\lambda(1)$.  For $x \in \R^d$ let $\mathring{x}(\lambda) \in C_\infty(\lambda)$ denote the closest point $C_\infty(\lambda)$ to $x$.  Let $d_\lambda$ denote the graph distance on $C_\infty(\lambda)$.  The constant $\mu(\lambda)$ is defined as the almost-sure limit $$\lim_{n \to \infty} \frac{d_\lambda(\mathring{0}(\lambda),\mathring{n}(\lambda))}{n} = \mu(\lambda)$$
    where we write $0 = (0,0,\ldots,0) \in \R^d$ and $n = (n,0,\ldots,0) \in \R^d$.  This limit is shown to exist by a subadditive ergodic theorem argument in \cite{YCG11} and so in fact we also have $\mu(\lambda) \leq \E d_\lambda(\mathring{0}(\lambda),\mathring{1}(\lambda))$. Note that if we take $\lambda \to \infty$ we almost surely have $d_\lambda(\mathring{0}(\lambda),\mathring{1}(\lambda)) \to 1.$  By the dominated convergence theorem we see that $\E d_\lambda(\mathring{0}(\lambda),\mathring{1}(\lambda)) \to 1$ completing the proof.
\end{proof}

We also need a non-trivial property of the time constant, namely that in each dimension $d$ the function $\lambda \mapsto \mu(\lambda)$ is continuous.  Recently, the first two authors proved Lipschitz continuity, and we refer to \cite{DG25} and the references therein for more context:

\begin{thm}[Theorem 1, Dubin and Gorski ~\cite{DG25}]\label{Continuity of Time Constant} 
    For each $d \geq 2$ and $\lam_0>\lam_c(d)$, there exists $C(\lam_0,d)$ such that 
    \[
       |\mu({\lam}) - \mu({\lam'})|
       \le C(\lam_0)|\lam-\lam'|
    \]
    for each $\lam,\lam' \in [\lam_0,\infty)$.
\end{thm}

    \subsection{Uniqueness of the giant component} \label{sec:uniqueness}

    For a random geometric graph $G$ we will write $\mathring{G}$ for the component of largest diameter in the underlying metric.\footnote{In the regimes we consider,
    this will with high probability coincide
    with the component of largest diameter in the \emph{graph} metric; for the most important case, see \cref{cor:graph-diam-giant}.}  Throughout, we will make use of the fact that there is only one component of large diameter.  We start with a version of this statement in Euclidean space for possibly distorted intensity measure and metric:

    \begin{lem} \label{lem:non-uniform-uniqueness}
        Fix $\lam_0 > \lam_c(d)$ and let $\phi_s \le s$ be an increasing
        function of $s$.  
        Then there exist $c(d,\lam_0) > 0, C = C(d) > 0$
        such that the following holds.

        Let $\sigma$ be a metric on $\R^d$
        which satisfies
        \[
            \sigma(x,y) \le \|x - y\| \le 2\sigma(x,y)
        \]
        for all $x,y \in \R^d$.
        Let  
        $\nu$ be a measure on
        $\R^d$ which satisfies 
        $\frac{d\nu}{d\mathcal{L}} \ge \lam_0$. 
        Let $G_s$ be the random geometric
        graph
        whose vertices are the
        points of a Poisson process on $[-s,s]^d$
        with intensity measure $\nu$
        and whose edge set consists
        of pairs of vertices with
        $\sigma$-distance at most 1.

        Then
        \[
        \Prob\left(
        \begin{array}{c}
            \mbox{Exactly one component of } G_s \\
            \mbox{ has Euclidean diameter at least } \phi_s
        \end{array}
        \right)
        \ge 1- C s^{d}\exp(-c \phi_s).
        \] 
    \end{lem}

    This is a slightly adapted version of \cite[Prop.~10.13]{Pen03}.  In \cite{Pen03}, there is no distortion and so we indicate how to adapt the proof of \cite[Prop.~10.13]{Pen03} to our setting in \cref{sec:appendix}.
    
    From here we will deduce that in our setting the giant component is unique with high probability.

    \begin{lem}\label{lem:giant-is-unique}
        Let $\lambda_0 > \lambda_c$.  Then there are constants $C = C(\lambda_0,M) > 0, c = c(\lambda_0,M) > 0$ so that for all $\lambda \geq \lambda_0$ and $r\phi \leq  c$ we have \begin{equation*}
            \Prob\left(
        \begin{array}{c}
            \mbox{Exactly one component of } G_M^n(r) \\
            \mbox{ has }d_M\mbox{-diameter at least } r\phi
        \end{array}
        \right) \geq 1- C r^{-d}\exp(-c\phi)\,.
        \end{equation*}
    \end{lem}
    \begin{proof}
        Let $\eta > 0$ be small enough so that $\lambda_0 e^{-\eta} > \lambda_c.$  For each $z \in M$ let $\delta_z \in (0,1)$ and $u_z$ be guaranteed from \cref{lem:almostflat}; we may further reduce $\delta_z$ so that $[-\delta_z, \delta_z]^d$ lies in the image of $u_z$.  Define $V_z = u_z^{-1}((-\delta_z,\delta_z)^d)$ and note that $\{V_z\}_{z \in M}$ forms an open cover of $M$.  By compactness of $M$ there is a finite set $\{z_1,\ldots,z_k\}$ so that $\mathcal{U} = \{V_{z_j}\}_{j = 1}^k$ is an open cover of $M$.
        
        Define the graph $H$ with vertex set $z_1,\ldots,z_k$ with edges defined by $E(H) = \{\{i,j\} : V_{z_i} \cap V_{z_j} \neq \emptyset\}.$  Note that since $M$ is connected, the graph $H$ is connected too. For each pair $\{i,j\} \in E(H)$ we may find a point $w_{i,j} \in V_{z_i} \cap V_{z_j}$ and $\eps_{i,j}$ so that $B_M(w_{i,j},\eps_{i,j}) \subset V_{z_i} \cap V_{z_j}.$  By potentially further shrinking $\eps_{i,j}$ we may define $U_{i,j} = u_{z_i}^{-1}(u_{z_i}(w_{i,j}) + [-\eps_{i,j},\eps_{i,j}]^d)$.  Let $G_i$ and $G_{i,j}$ denote the induced subgraph of $G_M^n(r)$ whose vertices lie in $V_{z_i}$ and $U_{i,j}$ respectively.  Let $\mathcal{E}_i$ denote the event that $G_i$ contains exactly one component of $d_M$-diameter at least $r\phi$; define $\mathcal{E}_{i,j}$ analogously.  Note that by \cref{lem:non-uniform-uniqueness} we have \begin{equation*}
            \P\left(\bigcap_{j \in [k]} \mathcal{E}_j \cap \bigcap_{\{i,j\} \in E(H)} \mathcal{E}_{i,j} \right) \geq 1 -  C r^{-d} \exp(-c \phi)
        \end{equation*}
        for some $C,c > 0$ depending on $\eta,\lambda_0,M$.  We will intersect with the event $\mathcal{E} = \bigcap_{j \in [k]} \mathcal{E}_j \cap \bigcap_{\{i,j\} \in E(H)} \mathcal{E}_{i,j}$ and let $\mathring{G}_i$ and $\mathring{G}_{i,j}$ denote the component of largest diameter.
        
        Note that for each $\{i,j\} \in E(H)$ we have that $U_{i,j} \subset V_{z_i} \cap V_{z_j}$ and so $\mathring{G}_{i,j} \subset \mathring{G}_i \cap \mathring{G}_j$.  In particular, for each edge $\{i,j\} \in E(H)$, the largest components in $V_{z_i}$ and $V_{z_j}$ are in fact the same component of $G_M^n(r)$.  Since $H$ is connected, this shows that the components $\mathring{G}_i$ all lie in the same component of $G_M^n(r)$ which we may call $\kappa$.  
        
        To see that $\kappa$ is the only component of large diameter,  let $\alpha > 0$ be the Lebesgue covering number of $\mathcal{U} = \{V_{z_j}\}_{j = 1}^k$ meaning that if $d_M(x,y) < \alpha$ then there is some $i \in [k]$ so that $x,y \in V_{z_i}$.  Note that if there is a component $\kappa'$ of $G_M^n(r)$ of $d_M$-diameter $\geq r\phi$, then there are points $x,y \in \kappa'$ so that $d_M(x,y) \in [r\phi,r(\phi+1)]$.  By restricting $c$, we have that $r(\phi + 1) \leq \alpha/2$ and so there is some $i \in [k]$ so that $x,y \in V_{z_i}$.  We note that either $x,y$ are connected in $G_i$ or otherwise the path from $x,y$ exits $V_{z_i}$ which would require having two points of distance, $\alpha/2$ apart.  On event $\mathcal{E}$, this implies that $\kappa' \cap G_i$ has metric diameter at least $r\phi$ implying that $\kappa' \subset \mathring{G}_i$. This shows that $\kappa' \subset \kappa$, completing the proof.
    \end{proof}

    Lastly, we note that, with high probability, $\mathring{G}^n_M(r)$ is also
    the component of largest diameter \emph{in the graph metric}.
    Thus, in our statements of \cref{thm:degree} and \cref{thm:main},
    we may interpret the phrase ``component of largest diameter''
    as referring to either $d_M$ or the graph metric, which
    will be important for algorithmic applications.
    \begin{cor}\label{cor:graph-diam-giant}
        Let $\lambda_0 > \lambda_c$. There exist constants $C=C(\lambda_0,M), c=c(\lambda,M)>0$ such that
        \[
            \Prob\left( \begin{array}{c}
                \mathring{G}_M^n(r) \mbox{ is the largest component} \\
                \mbox{ of } G_M^n(r) \mbox{ in graph diameter} 
                \end{array} \right) \ge 1 - C \exp(-c \log^2 n)
        \]
    \end{cor}
    \begin{proof}

    First, if $r \geq n^{-1/d +\beta}$ for some $\beta > 0$ then \cref{lem:connected} shows the graph is connected.  Otherwise, applying \cref{lem:giant-is-unique} twice shows that there is one component of $d_M$ diameter at least $c$, 
    and all other components have $d_M$-diameter larger at most $r \log^2 n$. Each component of $d_M$-diameter at most $r \log^2 n$ has graph diameter at most $(\log n)^{O_d(1)}$, while
    the giant component has graph diameter $\Omega(1/r)$.
    \end{proof}

\section{Squeezing between homogeneous processes in Euclidean space} \label{sec:setting-up-processes}

For each point $x \in M$, we will use \cref{lem:almostflat} to identify a small $d_M$-ball that can be identified with a subset of $\R^d$ up to small measure and metric distortion.  This will show that when we pushforward the measure $\lambda \cdot dM$ into $\R^d$ on this patch,  we may trap this measure between $\lowerlambda d\mathcal{L}$ and $\upperlambda d\mathcal{L}$ for $\lowerlambda$ and $\upperlambda$ close to $\lambda$.  For technical reasons, we will want that among all  $\lambda \in [\lambda_0,\infty)$ and $x \in M$ we have only a finite collection of choices for $\lowerlambda$ and $\upperlambda$. We begin by specifying these points and then define homogeneous random graphs in $\R^d$ whose vertex sets are Poisson with these intensities.

\subsection{Picking upper and lower intensities}

Throughout this subsection we will fix $\eps > 0$ and $\lambda_0 > \lambda_c$.  
Note that by \cref{lem:limit-time-constant}, there is some $\lambda_+$ so that 
\begin{equation}\label{eq:lam-+}
    \mu(\lambda_+) \leq e^{\eps/10}\,.
\end{equation}

For $\eta > 0$ to be chosen later as a function of $\eps,\lambda_0$, define $\lambda_k = \lambda_0 e^{(d + 1)\eta k}$ and define $K = \min\{k : \lambda_k \geq \lambda_+\}$.  By taking $\eta$ small enough, we may assume that $\lambda_{-1} > \lambda_c$.  Define the (multiplicative) net $\mathcal{N} = \{\lambda_k : k \in \{-1,0,\ldots,K+2\}\}$.  Now, for $\lambda \in [\lambda_0,\lambda_+ e^{(d+1)\eta}]$ define $$\lowerlambda =\lowerlambda(\lambda) = \max\{\lambda_k \in \mathcal{N} : \lambda_k \leq e^{-(d+1)\eta}\lambda\} \quad \text{ and } \quad  \upperlambda =\upperlambda(\lambda) = \min\{\lambda_k \in \mathcal{N} : \lambda_k \geq e^{(d+1)\eta}\lambda\}\,. $$
We will require two main features from these choices of $\lowerlambda$ and $\upperlambda$; first, we have \begin{equation}\label{eq:upper-lowerlam-bounds}
    \lowerlambda e^{{(d+1)\eta}} \leq \lambda \leq \upperlambda e^{(d+1)\eta}
\end{equation}
by construction.  We also have $\upperlambda e^{-2(d+1)\eta} \leq \lambda \leq \lowerlambda e^{2(d+1)\eta}$.  Thus, we may use the uniform continuity of $\lambda \mapsto \mu(\lambda)$ on the compact interval $[\lambda_{-1},\lambda_{K+2}] \subset (\lambda_c,\infty)$ guaranteed by \cref{Continuity of Time Constant} so that if we take $\eta$ small enough we have \begin{equation}\label{eq:upper-lower-TC-bounds}
    \frac{\mu(\lowerlambda)}{\mu(\lambda)} \leq e^{\eps/2} \,, \quad \frac{\mu(\lambda)}{\mu(\upperlambda)} \leq e^{\eps/2}\,,\quad \text{ and } \quad  \eta \leq \frac{\eps}{10}\,.
\end{equation}

For $\lambda \geq \lambda_+ e^{(d+1)\eta}$ we will simply define $\lowerlambda = \lowerlambda(\lambda) = \lambda_+ e^{-(d+1)\eta}$.  We will not make use of $\upperlambda$ for $\lambda \geq \lambda_+ e^{(d+1)\eta}$ and so we simply define $\upperlambda = \upperlambda(\lambda) = +\infty$ for $\lambda \geq \lambda_+$ so that way \eqref{eq:upper-lowerlam-bounds} holds  for all $\lambda \geq \lambda_0$.  Further, using \eqref{eq:lam-+} we have that \eqref{eq:upper-lower-TC-bounds} holds for all $\lambda \geq \lambda_0$ as well (where we interpret $\mu(+\infty) = 1$ using \cref{lem:limit-time-constant}).

\subsection{Specifying our homogeneous geometric random graphs in \texorpdfstring{$\R^d$}{Rd}}

Now that we have fixed our choices of $\lowerlambda$ and $\upperlambda$, we begin to define the processes that we trap $G_M^n$ between on a small patch.  For  $\eta > 0$ as chosen as a function of $\eps,\lambda_0$ above and for $x \in M$, let $u, \delta > 0$ be as guaranteed by \cref{lem:almostflat}.  We define $S := u(B_M(x,\delta))$ and 
\begin{equation} \label{eq:G-S-hat-def}
    G_S := u(G^{n}_M(r) \cap B_M(x,\delta)), \qquad 
    \Ghat_S := (e^{-\eta} r)^{-1} G_S\,.
\end{equation}
  
We want to identify $\Ghat_S$ as the restriction
of a geometric random graph on $\R^d$
to which \cref{lem:non-uniform-uniqueness}
will apply.   We first note how rescaling Euclidean space changes the intensity measure:

\begin{observation}\label{obs:rescale}
    For each $\lambda,r,t > 0$ and dimension $d$ we have $$t G_{\R^d}^\lambda(r) \equiv G_{\R^d}^{\lambda/t^d}(tr)\,.$$
\end{observation}
\begin{proof}
    This follows from seeing that if one takes an intensity $\lambda$ Poisson process on $\R^d$ and scales by $t$ then one obtains a $\lambda t^{-d}$ Poisson process in distribution.
\end{proof}

We now identify $\Ghat_S$ in terms of a metric and measure of bounded distortion.

\begin{lem}\label{lem:identify-G-S}
    There is a metric $\sigma$ on $\R^d$ and measure $\nu$ satisfying $$ \lambda e^{-(d+1)\eta}  \leq \frac{d\nu}{d\mathcal{L}} \leq \lambda e^{(d+1)\eta}\,,\qquad  1 \leq \frac{\|p - q\|}{\sigma(p,q)} \leq e^{2\eta} \quad \text{ for all }p,q \in (e^{-\eta} r)^{-1}S$$ so that if we set $G^\nu$ to be the geometric random graph on $\R^d$ whose vertex set is the Poisson point process of intensity measure $\nu$ and whose edges are between pairs $p,q$ with $\sigma(p,q) \leq 1$ then $$\widehat{G}_S \equiv G^\nu \cap (e^{-\eta} r)^{-1} S\,.$$
\end{lem}
\begin{proof}
    Define a metric $\sigma$ on $(e^{-\eta} r)^{-1}S 
    \subset \R^d$
    by taking
    \[
        \sigma(p,q) := \frac{1}{r}d_M(u^{-1}(e^{-\eta} rp), u^{-1}(e^{-\eta} rq)).
    \]
    Note that for all $p,q \in (e^{-\eta} r)^{-1} S$ we have
    \[
        \sigma(p,q) \le \|p - q\| \le e^{2\eta} \sigma(p,q),
    \]
     by \cref{lem:almostflat}. 
    We take an arbitrary extension of $\sigma$
    to all of $\R^d$ which still satisfies the 
    above inequalities.\footnote{Defining $\sigma$
    on all of $\R^d$ is not strictly necessary;
    one can state a version of \cref{non-uniform-Antal-Pisztora} where the underlying process
    is only defined on a subset of $\R^d$,
    but it is more complicated to state.
    One way to define such an extension is by
    \[
        \sigma(p,q) := \inf_{s \in S} \|p - s\| + \sigma(s,q)
    \]
    for all $p \in S, q \notin S$, and
    \[
        \sigma(p,q) := \min(\|p - q\|, \inf_{s_1,s_2 \in S}
        \|p - s_1\| + \sigma(s_1,s_2) + \|s_2 - q\|).
    \]}
        
    Next, define a measure $\nu$ on
    $\R^d$ by taking 
        \[
            \nu(A) :=
            \frac{\lam}{r^d} dM( u^{-1}((e^{-\eta}r) \cdot A))
        \]
        for $A \subset (e^{-\eta}r)^{-1}S$
        and setting
            $\nu(A) := \lam \mathcal{L}(A)$
        for $A \subset (e^{-\eta}r)^{-1}S^c$.\footnote{Again, the value of the measure outside of $(e^{-\eta}r)^{-1}S$
        will not play any role.}
        Then by \cref{lem:almostflat} we have \begin{equation*}
            \lambda e^{-(d+1)\eta}\mathcal{L}(A) \leq \nu(A) \leq \lambda e^{(d+1)\eta}\mathcal{L}(A) \,.
        \end{equation*}
        Define $G^\nu$ to be the random geometric graph on $\R^d$ whose vertex
        set is the Poisson point process
        with intensity measure $\nu$,
        and whose edges are between
        pairs $p,q$ of points with $\sigma(p,q) \le 1$.
        Then we see that
        \[
            \Ghat_S \equiv G^\nu \cap (e^{-\eta}r)^{-1}S,
        \]
        by \cref{obs:rescale}, 
        as desired.
\end{proof}

Recalling our choice of $\lowerlambda$ and $\upperlambda$, define 

\begin{align}\label{eq:Ghat-defs}
    \Ghat^+ = G_{\R^d}^{\upperlambda}(1) \cap ((e^{\eta} r)^{-1} \cdot S) \quad &\text{ and } \quad \Ghat^- = G_{\R^d}^{\lowerlambda}(1) \cap ((e^{-\eta} r)^{-1} \cdot S) \\
    G^+(r) = e^\eta r \Ghat^+ \quad &\text{ and } \quad G^-(r) = e^{-\eta} r \Ghat^-\,. \nonumber 
\end{align}

By \cref{lem:identify-G-S} we may couple the underlying Poisson processes so that we have the inclusions \begin{equation}\label{eq:ghat-inclusions}
    G^-(r) \subset G_S \subset G^+(r)\,.
\end{equation}

\section{Moving from one giant component to another} \label{sec:move-giant-component}

A crucial step towards \cref{thm:main} is to show that the largest component of $G_M^n(r)$ contains the largest component of the graph $G^-$ defined in \cref{sec:setting-up-processes} for each point $x \in M$.  This section shows how we can compare giant components.  We also need to show that each point in $M$ is close to the giant component.   To begin with, recall that $\mathring{G}_M^n(r)$ is the component of $G_M^n(r)$ of largest $d_M$ diameter and that for $y \in M$ we let $\mathring{y}$ denote the $d_M$-closest point in $\mathring{G}_M^n(r)$ to $y$.  The first main goal of this section is the following lemma:

\begin{lem}\label{lem:all-near-giant}
    Fix $\beta >0$ and $\lambda_0 > \lambda_c$.      There are constants $c = c(\lambda_0,d,\beta,M) > 0$ and $n_0 = n_0(\lambda_0,d,\beta,M)$ so that if $r \leq n^{-\beta}$ and $\lambda \geq \lambda_0$ then the event \begin{equation*} \mathcal{B} = \left\{\text{for all } y \in M, d_M(y,\mathring{y}) \leq r \log^2 n \right\}
    \end{equation*}
    satisfies $\P(\mathcal{B}^c) \leq \exp(-c \log^2 n)$ for all $n \geq n_0$.
\end{lem}

We also will need a more technical lemma, which will say that if we have a point $x \in M$ and let $G^-$ be as defined in \cref{sec:setting-up-processes} then there is some component $\kappa$ of $G^-$ so all points in the global giant component $\mathring{G}_M^n(r)$ near $x$ are close to $\kappa$ in the graph metric.  Write $d_r^\lambda$ for the graph distance on $G_M^n(r).$

\begin{lem} \label{sparse-giant-nearby}
    Fix $\beta,\eps >0, \lambda_0 > \lambda_c$ and $x \in M$.      There are constants $c = c(\lambda_0,d,\eps,\beta,M) > 0$ and $n_0 = n_0(\lambda_0,d,\eps,\beta,M), L' = L'(\lambda_0,d,\eps)$ so that if $r \leq n^{-\beta}$ and $\lambda \geq \lambda_0$ then the following holds. 
    The event
    \[
    \cC:=
    \left\{
        \begin{array}{c}
        \text{there exists a connected component }
        \kappa \text{ of } G^-(r) \\
        \text{such that }
        \forall p \in \mathring{G}^{n}_M(r) \cap B_M(x,\delta/L'),
        \exists p' \in \kappa  \\ 
        \mbox{ such that }
        d^\lam_r(p,u^{-1}(p')) \le \log^2 n
        \end{array}
    \right\}
    \]
    satisfies
    $\Prob(\cC^c) \leq \exp(- c \log ^2n)$ for all $n \geq n_0$.
\end{lem}

\cref{sparse-giant-nearby} is the more complicated event and we will essentially prove \cref{lem:all-near-giant} along the way.  We first define three intermediate events.

Recalling the definition of $\lambda_{-1}$ from \cref{sec:setting-up-processes}, let $L = 10\max\{\mu(\lambda_{-1}),d\}$.  Define  $\kappa$ to be the connected
component of $G^-(r)$ containing
the connected component of 
$G^-(r) \cap [-\delta/L,\delta/L]^d$ of largest diameter.  For a parameter $\gamma \in [(\log n)^{-1/2},1]$ define the following events:
    \begin{gather*}
        \cC_1:=
        \left\{ 
        \forall p \in  B(0,\delta/2L),
        \exists p' \in \kappa \mbox{ such that } 
        \|p - p'\| \le \gamma r \log^2 n 
        \right\},\\
        \cC_2:=
        \left\{ 
            \forall p \in \kappa
            \cap B(0,\delta/L),
            \forall q \in 
            u(\mathring{G}^{n}_M(r) 
            \cap B_M(x,\delta/2L)),
            p \xleftrightarrow{G_S} q
        \right\},\\
        \cC_3:=
        \left\{ \begin{array}{c}
            \forall p,q \in G_S \cap u(B_M(x,\delta/L))
            \mbox{ such that } 
            p \xleftrightarrow{G_S} q\\
            \text{ and } \|p-q\| \le  \gamma r \log^2 n,
            \mbox{ we have } d_{G_S}(p,q) \le 
             \log^2 n
        \end{array} \right\}.
    \end{gather*}

    The event $\cC_1$ states that all points in $B(0,\delta/2L)$ are near the component $\kappa$; $\cC_2$ states that each point of the global giant $\mathring{G}_M^n(r)$ near $x$ is connected to $\kappa$ once we map to Euclidean space by $u$; and finally $\cC_3$ states that nearby points connected in $G_S$ cannot have too large graph distance between them.  We will ultimately see that $\cC_1 \cap \cC_2 \cap \cC_3 \subset \cC$ for $n$ sufficiently large. 

    We prove each event occurs with high probability in its own subsection, and introduce relevant facts about percolation when needed.

\subsection{The event \texorpdfstring{$\cC_1$}{C1}: all points are close to \texorpdfstring{$\kappa$}{kappa}}
We show $\cC_1$ occurs with high probability.  

\begin{lem} \label{P(C)-large}
    In the context of \cref{sparse-giant-nearby}, there is $n_0 = n_0(\lambda_0,d,\eps,\beta)$ and $c = c(\lambda_0,d,\eps,\beta) > 0$ so that $\Prob(\cC_1^c) \leq \exp(-c \gamma \log^2 n)$ for $n \geq n_0$.
\end{lem}

This requires a piece of the literature from continuum percolation which states that all points are close to the (almost-surely) unique infinite cluster:

\begin{lem}[Lemma 3.3, Yao, Chen, and Guo~\cite{YCG11}] \label{Yao o close in infinite comp}
    Suppose $\lambda > \lambda_c$ and let $C_\infty$ denote the (almost-surely) unique infinite component of $G_{\R^d}^\lambda(1)$. Then there exists a constant $c = c_\lambda>0$ such that for all $R \geq c^{-1}$ we have 
    \[
        \Prob(B_{\R^d}(0,R) \cap C_\infty = \emptyset) \le \exp(-c R^{d-1}).
    \]  
\end{lem}    

We are now ready to handle the event $\cC_1$:
\begin{proof}[Proof of \cref{P(C)-large}]
First, consider the uniqueness event
\begin{equation*}
U_r^- = \left\{
\begin{array}{c}
\text{at most one component of the graph }
G^-(r) \cap [-\delta/L,\delta/L]^d \\
\text{has Euclidean diameter }\geq r \gamma \log^2n  
\end{array}
\right\}\,.
\end{equation*}
Note that by our choice of $L$
we have $[-\delta/L,\delta/L]^d \subset B_{\R^d}(0,e^{-\eta} \delta) \subset S$.
Therefore, 
\[G^-(r) \cap [-\delta/L,\delta/L]^d
=
(e^{-\eta} r)G^{\lowerlambda}_{\R^d} \cap 
[-\delta/L,\delta/L]^d, \]
and so $$\Prob((U^-_r)^c) \le  \exp(-c \gamma \log^2 n)$$  for $c = c(d,\lambda_0,\beta) > 0$ and $n \geq n_0(d,\lambda_0,\beta)$ by \cref{lem:non-uniform-uniqueness}.

Now, let $C_\infty^-$ be
the unique infinite component
of $G^{\lowerlambda}_{\R^d}$.
Note that whenever $r>0$ is
sufficiently small that
$2 \gamma r \log^2 n  \le \delta/2L$, we 
have 
\begin{equation}\label{eq:move-to-Cinf}
    U^-_r \cap
    \left\{ (e^{-\eta}r)C_\infty^- \cap B_{\R^d}\left(0,\frac{\delta}{2L}
    +  \gamma r \log^2n \right)
    \ne \emptyset \right\}
    \subset 
    \left\{
        (e^{-\eta}r)C_\infty^- \cap B_{\R^d}\left(0,\frac{\delta}{2L} 
        +  \gamma r \log^2 n \right)
        \subset \kappa
    \right\}.
\end{equation}
This is because any component
of $(e^{-\eta} r)C_\infty^- \cap [-\delta/L,\delta/L]^d$
which intersects $B_{\R^d}(0,(\delta/2L) + \gamma r \log^2n)$
has Euclidean
diameter at least 
$\delta/2L - \gamma r \log^2n \ge \gamma r \log^2n $;
$U^-_r$ then implies that such a component
must lie in $\kappa$.
Thus, on $U^-_r$, to find a nearby point
in $\kappa$, it suffices to find a
nearby point in $(e^{-\eta}r)C^-_\infty$.

Therefore, using \eqref{eq:move-to-Cinf}, \cref{obs:rescale}, and 
\cref{Yao o close in infinite comp} 
we have
\begin{align*}
    \Prob(\cC_1^c \cap U^-_r)
    &\le
    \Prob\left(
    \begin{array}{c}
        \exists p \in B_{\R^d}(0, \delta/(2Le^{-\eta}r))
        \text{ such that } \forall p' \in C_\infty^-,\\
        \|p - p'\| >  e^\eta \gamma \log^2n
    \end{array}
    \right) \\
    &\le
    \sum_{\substack{p \in \Z^d \\ \|p\| \le \delta/(2Le^{-\eta}r)}}
    \Prob\left( 
    B_{\R^d}(p, \gamma e^\eta \log^2n - \sqrt{d}/2)
    \cap C^-_\infty = \emptyset
    \right) \\
    &\le C L^d r^{-d} \exp(-c'(\gamma \log^2 n)^{d - 1})
\end{align*}
for $c' = c'(\lowerlambda) > 0$ and $C = C(d) > 0$.  Since $d \geq 2$, 
for $n \geq n_0(\lowerlambda,d,\lambda_0,\beta)$ and some $c'' =  c''(\lowerlambda,d,\lambda_0,\beta)$ we have
$\Prob(\cC_1^c) \le \Prob(\cC_1^c \cap U^-_r)
+ \Prob((U^-_r)^c)
\le \exp(-c'' \gamma \log^2 n)$,
as desired. 
\end{proof}

\subsection{The event \texorpdfstring{$\cC_2$}{C2}: the global giant contains a chunk of \texorpdfstring{$\kappa$}{kappa}}

We now move on to handling $\cC_2$, which will follow from \cref{lem:non-uniform-uniqueness} and \cref{P(C)-large}.

\begin{lem} \label{P(E)-large}
In the context of \cref{sparse-giant-nearby}, there is $n_0 = n_0(\lambda_0,d,\eps,\beta)$ and $c = c(\lambda_0,d,\eps,\beta) > 0$ so that $\Prob(\cC_2^c) \le \exp(-c\log^2 n)$ for $n \geq n_0$.
\end{lem}
\begin{proof}
    Consider the following uniqueness event on $G_S$:
    \[
        U :=
        \left\{
        \begin{array}{c}
        \mbox{ at most one component of }
        G_S \cap [ -\delta/L, \delta/L ]^d \\ \mbox{ has Euclidean diameter 
        at least } r \log^2 n
        \end{array}
        \right\},
    \]
    Again, our choice of $L$
    ensures that 
    \[ G_S \cap [ -\delta/L, \delta/L ]^d =
    G^\nu \cap [ -\delta/L, \delta/L ]^d\]
    and therefore $\Prob(U^c) \le \exp(-c \log^2n)$ for $c = c(d,\lambda_0,\beta) > 0$ and $n \geq n_0(d,\lambda_0,\beta)$ by \cref{lem:non-uniform-uniqueness} where we use \cref{lem:identify-G-S} to verify the hypotheses. 

    Now we claim that, for $r>0$ sufficiently
    small, we have
    \begin{equation} \label{eq:LB-E}
        U \cap \cC_1 \subset \cC_2
    \end{equation}
    where we take $\gamma = 1$ for this instance of $\cC_1$. 
    Assume that $U \cap \cC_1$ holds,
    and let
    $p \in \kappa \cap B_{\R^d}(0, \delta/L)$
    and 
    $q \in u\left(\mathring{G}_M^{n}(r)
    \cap B_M(x, \delta/2 L)\right)$.
    If no such $p$ and $q$ exist then
    $\cC_2$ holds vacuously.
    Recall that $G^-(r) \subset G_S$,
    and so there is a unique 
    connected component $\kappa'$
    of $G_S \cap [-\delta/L,\delta/L]^d$ containing the largest connected
    component of $G^- \cap [-\delta/L,\delta/L]^d$. The event 
    $U \cap \cC_1$ implies that 
    the Euclidean diameter of $\kappa'$
    is at least
    $\delta/L - 2 \gamma r \log^2 n \geq r \log^2n$,
    for $n$ sufficiently large.
    By $U$, $\kappa'$ is the \emph{unique}
    connected component of
    $G_S \cap [-\delta/L,\delta/L]^d$
    of Euclidean diameter at least $r \log^2 n$.

    With this in mind, let $\xi$ be the connected component of $u\left(\mathring{G}_M^{n}(r)
    \cap B_M(x, \delta)\right) \cap 
    [-\delta/L, \delta/L]^d$
    containing $q$.  To prove that $p \xleftrightarrow{G_S} q$, it is sufficient to show that $\xi$ has Euclidean diameter at least $r \log^2 n$. 
    For $n$ sufficiently large we have that
    $r \log^2 n < \delta/2L$.  Thus if the {Euclidean} diameter of $\xi$ is at most $r \log^2 n$ then $\xi \subset B(q,r \log^2n)$; since $u^{-1}(\xi) \subset \mathring{G}_M^n(r)$, this would imply that $u^{-1}(\xi) = \mathring{G}^{n}_M(r)$.  
    But then the $d^\lam_r$-diameter
    of $\mathring{G}_M^{n}(r)$
    is at most $O(\log^{2d}n)$,
    which contradicts the
    definition of $\mathring{G}_M^{n}(r)$
    as the largest component of $G_M^{n}(r)$, since the component of $G_M^{n}(r)$
    containing $\kappa$ has
    $d^\lam_r$-diameter at least 
    $\Omega(1/r)$ by event $\cC_1$.  This shows \eqref{eq:LB-E}, and so $\Prob(\cC_2^c) \le \Prob(U^c) + \Prob(\cC_1^c)
    = O(\exp(-c \log^2 n))$ by \cref{P(C)-large}.
\end{proof}

\subsection{The event \texorpdfstring{$\cC_3$}{C3}: geodesics do not wander too far}

Our final step before proving \cref{sparse-giant-nearby} is to handle the event $\cC_3$:

\begin{lem} \label{P(D)-large} In the context of \cref{sparse-giant-nearby}, there is $n_0 = n_0(\lambda_0,d,\eps,\beta)$, $c = c(\lambda_0,d,\eps,\beta) > 0$ and $\gamma = \gamma(\lambda_0,d)$ so that
    $\Prob(\cC_3^c) \le \exp(-c \log^2n)$ for $n \geq n_0$.
\end{lem}

This follows quickly from a continuous version of the Antal-Pisztora theorem \cite{AP96} proved in \cite[Lem.~3.4]{YCG11}:

\begin{lem} \label{non-uniform-Antal-Pisztora}
    For each $\lam_0 > \lam_c(d)$ there is $c(d,\lam_0) > 0, \rho(d,\lam_0) < \infty$
    such that the following holds.

    Let $\sigma$ be a metric on $\R^d$
    which satisfies
    \[
        \sigma(x,y) \le \|x - y\| \le 2\sigma(x,y)
    \]
    for all $x,y \in \R^d$.
    Let $\nu$ be a positive measure on
    $\R^d$ which satisfies 
    $\frac{d\nu}{d\mathcal{L}} \ge \lam_0$.
    Let $G$ be the random geometric
    graph
    whose vertices are the
    points of a Poisson process on $\R^d$
    with intensity measure $\nu$
    and whose edge set consists
    of pairs of vertices with
    $\sigma$-distance at most 1.

    Then for any $x,y \in \R^d$
    and any $t \ge \rho\|x-y\|$ we have
    \[
        \Prob_{x,y}( x \stackrel{G}{\longleftrightarrow} y, d_G(x,y) \ge t)
        \le \exp(-ct).
    \]
\end{lem}

The stated version of \cite[Lemma 3.4]{YCG11} does not allow for our required inhomogeneity and so we prove \cref{non-uniform-Antal-Pisztora} in \cref{sec:appendix}.

\begin{proof}[Proof of \cref{P(D)-large}]
    Recall that $\Ghat_S = G^\nu \cap (e^{-\eta}r)^{-1} S$ where \cref{lem:identify-G-S} allows us to identify $\Ghat_S$ in terms of a distorted metric and measure.  Therefore, using \cref{Palm theory} and \cref{non-uniform-Antal-Pisztora} and $\gamma \leq (10\rho)^{-1}$
    allows us to bound
    \begin{equation*}
        \Prob(\cC_3^c)
        \le \left[\frac{\lam e^{d\eta}}{r^d}\mathrm{Vol}_M(B_M(x,\delta))\right]^2
        \exp(-c \log^2n) \leq \exp(-c'\log^2 n)\,. \qedhere
    \end{equation*}
\end{proof}

\subsection{Putting the pieces together: proofs of \texorpdfstring{\cref{sparse-giant-nearby}}{Lemma 18} and \texorpdfstring{\cref{lem:all-near-giant}}{Lemma 17}}
We start with the proof of \cref{sparse-giant-nearby}:

\begin{proof}[Proof of \cref{sparse-giant-nearby}]
    Choose
    $L' := 4 L$ and $\gamma$ so that \cref{P(D)-large} holds.      Using the fact that
    for all $p,q \in \mathring{G}^{n}_M(r) \cap B_M(x,\delta)$ we have $d^\lam_r(p,q) {\leq } d_{G_S}(u(p), u(q))$ 
    and the fact that $u$
    has distortion at most $e^\eta$, we claim that 
    \begin{equation} \label{eq:lower-bound-B}
        \cC_1 \cap \cC_2 \cap \cC_3  \subset \cC
    \end{equation}
    for $n$ sufficiently large.  To see this, let $p \in \mathring{G}_M^n(r) \cap B_M(x,\delta/L')$.  By $\cC_1$, there is some $p' \in \kappa$ so that $\|u(p) - p'\| \leq \gamma r \log^2 n$.  By $\cC_2$ we have $p' \xleftrightarrow{G_S} u(p)$.  By $\cC_3$ we have $d_r^\lam(p,u^{-1}(p')) \leq d_{G_S}(p',u(p)) \leq \log^2 n.$ 
    The desired bound on $\Prob(\cC^c)$
    follows from combining \eqref{eq:lower-bound-B} with \cref{P(C)-large}, \cref{P(E)-large}, and \cref{P(D)-large}.
 \end{proof}

We now prove \cref{lem:all-near-giant}, which essentially follows from \cref{P(C)-large}.

\begin{proof}[Proof of \cref{lem:all-near-giant}]
    Recall that $\mathring{y}$ is equal to the (almost-surely unique) $z \in \mathring{G}_M^n(r)$ so that $d_M(y,z) = d_M(y,\mathring{G}_M^n(r))$.  It will be sufficient to show that for all $y \in M$ there is some point $z \in \mathring{G}_M^n(r)$ so that $d_M(z,y) \leq r \log^2 n.$  Since $r \geq \Omega_{d,M}(n^{-1/d})$ we may find an $r$-net $\mathcal{S} \subset M$ in the metric $d_M$ of size $|\mathcal{S}| \leq n^{2}$, say, for $n$ sufficiently large.  We then have \begin{align} \label{eq:net-for-near-giant}
        \P\left(\exists~y \in M:d_M(y,\mathring{y}) \geq r \log^2 n  \right) \leq \sum_{y \in \mathcal{S}} \P\left(d_M(y,\mathring{y}) \geq (1/2)r \log^2 n  \right)\,.
    \end{align}
    Choose $\lowerlambda$ and $\eta > 0$ via \cref{sec:setting-up-processes} for some fixed $\eps \in (0,1)$.  We may then use compactness of $M$ and \cref{lem:almostflat} to choose a finite set of points $x_1,\ldots,x_k$ along with maps $u_{j}$ and radii $\delta_{j}$ so that $\bigcup_{j = 1}^k B_M(x_j, \delta_j) = M$ and $u_j,\delta_j$ satisfy the conclusion of \cref{lem:almostflat}.  By \cref{P(C)-large} along with \cref{lem:giant-is-unique} we have that $$\P\left(d_M(y,\mathring{y}) \geq (1/2)r \log^2 n \right) \leq \exp(-\Omega_{M,\beta}( \log^2 n))$$ for all $n$ sufficiently large.   Combining with \eqref{eq:net-for-near-giant} completes the proof.
\end{proof}

\section{Metric comparison bounds on patches} \label{sec:patch}

In this section we prove our main technical lemma towards \cref{thm:main}.  For a given point $x \in M$ and parameters $\delta,\eps > 0$ define the event $$\mathcal{G}_{\delta,\eps}(x) = \left\{\forall p, q \in \mathring{G}^{n}_M(r) \cap B_M(x,\delta)
\mbox{ such that } d_M(p,q) \ge r \log^3 n, \frac{r d^\lam_r(p,q)}{\mu(\lam) d_M(p,q)} \in 
[e^{-\epsilon},e^\epsilon]  \right\}\,. $$

This event indicates that the metric $d_r^\lambda$ on $\mathring{G}_M^n(r)$ near $x \in M$ looks locally like $\mu(\lambda) d_M$.  We show it holds with high probability.

\begin{lem} \label{lem:patches}
    Fix $\beta,\eps >0, \lambda_0 > \lambda_c$ and $x \in M$.  
    There are constants $c = c(\lambda_0,d,\eps,\beta) > 0$, $n_0 = n_0(\lambda_0,d,\eps,\beta)$ and $\delta' = \delta'(x,\eps,\lambda_0,d)$ so that if $r \leq n^{-\beta}$ and $\lambda \geq \lambda_0$ then  $$\P(\mathcal{G}_{\delta',\eps}(x)) \geq 1 - \exp(-c \log^2n)\,.$$
\end{lem}

The main work is to first show that we can apply the quantitative shape theorem, \cref{large-deviations-time-constant}, and use the event $\cC$ from \cref{sparse-giant-nearby} to move from the local giant component to the global giant component.

\subsection{Applying the shape theorem locally}\label{sec:localshape}

For a fixed $x \in M, \lambda_0 > \lambda_c$ and $\eps > 0$, recall that we have chosen $\delta > 0$ and $\eta > 0$ in \cref{sec:setting-up-processes}.  For all $\lambda \geq \lambda_0$ we have defined the graphs $G^{\pm}(r)$, their rescaled versions $\Ghat^{\pm}$, and the values $\lowerlambda, \upperlambda$ in \cref{sec:setting-up-processes}.  Crucially, we recall that for each $\lambda_0$ and $\eps > 0$, there are only finitely many choices for $\lowerlambda$ and $\upperlambda$.

We will show a quantitative version of the 
shape theorem holds for $G^\pm(r)$ by applying \cref{large-deviations-time-constant} to the rescaled graphs $\Ghat^\pm$. 
Let $d_r^{\pm}$ be the standard graph distance on $G^{\pm}(r)$.  Define $L = 10 \mu(\lambda_{-1})$ where $\lambda_{-1}$ is defined in \cref{sec:setting-up-processes}. 
Define the event
\begin{equation} \label{eq:A_r-def}
    A^- = \left\{\begin{array}{c}
    \forall~p,q \in G^-(r) \text{ with } d_r^-(p,q) < \infty \text{, }
    \|p\|,\|q\| \le \delta/L
    \text{, and }  \\
    \|p - q\| \geq e^{-\eta} r \log^2n \text{, we have } \frac{ rd_r^-(p,q)}{\|p - q\|} \leq e^{2\eta} \mu(\lowerlambda)
      \end{array} \right\}\,.
\end{equation}
For $\lambda \in [\lambda_0,\lambda_+ e^{(d+1)\eta}]$ (where we recall $\lambda_+$ is defined at \eqref{eq:lam-+}) define

\begin{equation} \label{eq:A_r+def}
     A^+ = \left\{\begin{array}{c}
    \forall~p,q \in G^+(r) \text{ with } d_r^+(p,q) < \infty \text{, }
    \|p\|,\|q\| \le \delta/L
    \text{, and }  \\
    \|p - q\| \geq e^{-\eta} r \log^2n \text{, we have } \frac{ rd_r^+(p,q)}{\|p - q\|} \geq e^{-2\eta} \mu({\upperlambda})
    \end{array} \right\} 
\end{equation}
while for $\lambda > \lambda_+ e^{(d+1)\eta}$ note that by \eqref{eq:lam-+} we deterministically have  
\begin{equation}\label{eq:A+-large-lam}
	\frac{r d_r^+(p,q)}{\|p - q\|} \geq e^{-\eta} \geq e^{-\eps/5} \mu(\upperlambda)
\end{equation}
and so for convenience for $\lambda > \lambda_+ e^{(d+1)\eta}$ we define $A^+$ to be any almost-sure event.
The main step in this subsection is to prove that $A^{\pm}$ hold with high probability.

\begin{lem}\label{lem:Ar-events}
    Fix $\beta,\eps >0, \lambda_0 > \lambda_c$ and $x \in M$.  
    There are constants $c = c(\lambda_0,d,\eps,\beta,M) > 0$ and $n_0 = n_0(\lambda_0,d,\eps,\beta,M)$ so that if $r \leq n^{-\beta}$ and $\lambda \geq \lambda_0$ then  $\P(A^{\pm}) \geq 1 - \exp(-c \log^2 n)$ for all $n \geq n_0$.
\end{lem}
\begin{proof}
    We handle the bound on $A^-$ and then indicate how the bound on $A^+$ follows similarly. We first claim that 
    \begin{align}
    (A^-)^c &\subset
    \left\{
    \begin{array}{c}
        \exists p,q \in \Ghat^-
        \text{ with } 
        d_{\Ghat^-}(p,q) < \infty \text{, } 
        \|p\|,\|q\| \le \frac{\delta}{L}(e^{-\eta}r)^{-1}, \\
        \|p - q\| \geq \log^2n,
        \text{ and } \frac{d_{\Ghat^-}(p,q)}{{\mu(\lowerlambda)}\|p - q\|}
        \notin [e^{-\eta},e^\eta]
    \end{array}
    \right\} \nonumber 
    \\
    &\subset \left\{
    \begin{array}{c}
        \exists p,q \in G^{\lowerlambda}_{\R^d}(1)
        \text{ with } 
        d_{G^{\lowerlambda}_{\R^d}}(p,q) < \infty \text{, } 
        \|p\|,\|q\| \le \frac{\delta}{L}(e^{-\eta}r)^{-1}, \\
        \|p - q\| \geq \log^2n,
        \text{ and } \frac{d_{G^{\lowerlambda}_{\R^d}}(p,q)}{{\mu(\lowerlambda)}\|p - q\|}
        \notin [e^{-\eta},e^\eta] \label{eq:Arminus-bound}
    \end{array}
    \right\}
    \end{align}
    where we note that first inclusion follows from $G^-(r) = (e^{-\eta}r)\Ghat^-$ by \eqref{eq:Ghat-defs}.  To see the second inclusion,  suppose 
    that $p,q \in \Ghat^- \cap B(0, (\delta/L)(e^{-\eta}r)^{-1})$ satisfy $\|p-q\| \ge \log^2 n$,
    $d_{\Ghat^-}(p,q) < \infty$,
    and $d_{\Ghat^-}(p,q)/\|p-q\| \notin [e^{-\eta},e^\eta]$.
    If a $G^{\lowerlambda}_{\R^d}$-geodesic from $p$ to $q$ lies
    within ${(e^{-\eta} r)^{-1}} \cdot S$, then we have
    \[
        d_{G^{\lowerlambda}_{\R^d}}(p,q)
        = d_{\Ghat^-}(p,q) \notin [e^{-\eta} \mu(\lowerlambda)\|p-q\|,e^\eta  \mu(\lowerlambda)\|p-q\|].
    \]
    On the other hand, if a $G^{\lowerlambda}_{\R^d}$-geodesic
    leaves ${(e^{-\eta} r)^{-1}} \cdot S$, it contains at least two disjoint subpaths
    connecting $B(0, (\delta/L)(e^{-\eta}r)^{-1})$
    to $((e^{-\eta} r)^{-1} \cdot S)^c \subset 
    ((e^{-\eta} r)^{-1} \cdot
    B(0, e^{-\eta} \delta ))^c$,
    and so we have 
    \begin{equation}\label{eq:if-geodesic-wanders}
        d_{G^{\lowerlambda}_{\R^d}}(p,q) 
        \ge 2\left(e^{-\eta}\delta - \frac{\delta}{L}\right)
        (e^{-\eta} r)^{-1}
        \ge e^{\eta} \mu(\lowerlambda) \frac{2\delta}{L}
        (e^{-\eta} r)^{-1}
        \ge e^\eta \mu(\lowerlambda)\|p-q\|
    \end{equation}
    where the second inequality follows from our
    choice of $L$ along with \cref{lem:limit-time-constant}.
    Thus, we have established \eqref{eq:Arminus-bound}.
     
    Therefore, by \cref{large-deviations-time-constant} and \cref{Palm theory} we may bound 
    \begin{align*}
        \P((A^-)^c) &\leq 
        \E\left[ 
        \sum_{\substack{p,q \in G^{\lowerlambda}_{\R^d}(1) 
        \cap B(0,r^{-1}): 
        \|p - q\| \geq  \log^2 n}} \one\left\{
        d_{G^{\lowerlambda}_{\R^d}}(p,q)/\|p - q\|
        \notin [e^{-\eta}\mu(\lowerlambda),e^\eta\mu(\lowerlambda)] \cup \{\infty\}
        \right\} \right] \\
        &\leq C_d \lowerlambda^2 r^{-d}  
        \exp(-c_{\lowerlambda,\eta,d} \log^2n)\,.
    \end{align*}
	Recalling that for fixed $\eps > 0,\lambda_0 > \lambda_c$ there are only finitely many choices for $\lowerlambda$ and $r \geq \Omega_d(n^{-1/d})$,  we may bound $$ C_d \lowerlambda^2 r^{-d}  
        \exp(-c_{\lowerlambda,\eta,d} \log^2n) \leq  \exp(-c_{\lambda_0,\eps,d,\beta}' \log^2n )$$
        for $n \geq n_0(\lambda_0,\eps,d,\beta)$.  This 
   completes the bound on $A^-$.
        The proof for $A^+$ is exactly analogous.
\end{proof}

\subsection{Proof of \texorpdfstring{\cref{lem:patches}}{Lemma 24}}
We are now ready to combine the pieces to prove \cref{lem:patches}.

\begin{proof}[Proof of \cref{lem:patches}]
    For $x \in M, \lambda_0 > \lambda_c$ and $\eps > 0$, we have chosen parameters $\delta > 0$ and $\eta > 0$ in \cref{sec:setting-up-processes}; further, we have also chosen $\lowerlambda$ and $\upperlambda$ for all $\lambda \geq \lambda_0$ in addition to the map $u$ from \cref{lem:almostflat} with small distortion.  
    With these choices, we proceed by showing $A^+ \cap A^- \cap \cC
        \subset
        \mathcal{G}_{\delta/L',\epsilon}(x)$
    and then our desired result
    will follow from
    \cref{lem:Ar-events}
    and \cref{sparse-giant-nearby}.  
    So assume $A^+, A^-,$ and $\cC$ all hold.
    Let $p, q \in \mathring{G}^{n}_M(r) \cap B_M(x,\delta/L')$ with $d_M(p,q) \ge r \log^3 n$.
    We want to show that 
    \begin{equation} \label{eq:M-local-need} e^{-\eps} \leq 
        \frac{r d^\lam_r(p,q)}{\mu(\lam)d_M(p,q)} \leq e^\eps
    \end{equation}
    We first show the lower bound.  Note that by the distortion bound on $u$ from \cref{lem:almostflat} we have 
    $\|u(p) - u(q)\| \ge e^{-\eta}r \log^3 n$.
    First consider the case where $d^\lam_r(p,q) < d_{G_S}(u(p),u(q))$; in this case, we have that each geodesic from $p$ to $q$ must wander far enough to contain a vertex outside of $B_M(x,\delta)$.  We must have $rd^\lam_r(p,q) \geq 2\delta(1 - 1/L')$.  Since $d_{M}(p,q) \leq 2\delta/L'$, by the choice of $L'$ we have $r d_r^\lam(p,q) \geq e^{-\eps} \mu(\lambda) d_M(p,q).$  Thus we may assume $d_r^\lam(p,q) = d_{G_S}(u(p),u(q))$.    
    Now, since $u(G^n_M(r) \cap B_M(x,\delta)) = G_S \subset G^+$, we have $u(p), u(q) \in G^+$
    and 
    \[
        d^\lam_r(p,q) = d_{G_S}(u(p),u(q))
        \ge d^+_r(u(p),u(q)) \,.
    \]
    Recall that we have defined $A^+$ separately in the case of $\lambda \geq \lambda_+e^{(d+1)\eta}$, in which case we note that the lower bound in \eqref{eq:M-local-need} follows immediately from \eqref{eq:upper-lower-TC-bounds} and \eqref{eq:A+-large-lam}.  In the case of $\lambda \leq \lambda_+e^{(d+1)\eta}$ we may use $A^+$, 
    the distortion bound on $u$ and \eqref{eq:upper-lower-TC-bounds} to see 
    \begin{align*}
       rd^{\lam}_r(p,q) &\ge rd^+_r(u(p),u(q)) \ge e^{-2\eta} \mu(\upperlambda)\|u(p) - u(q)\|  \ge e^{-3\eta - (\epsilon/2)} \mu(\lam) d_M(p,q)
       \ge e^{-\epsilon} {\mu(\lam)}d_M(p,q)
    \end{align*}
    as desired.

    For the upper bound in \eqref{eq:M-local-need}, use $\cC$
    to find $p',q'$ in the same component
    of $G^-$ satisfying
    \begin{equation}\label{eq:B-application}
        d^\lam_r(p, u^{-1}(p')) \leq \log^2n \quad \text{ and } \quad d^\lam_r(q, u^{-1}(q'))
    \le \log^2n\,.
    \end{equation}
    
    Since $G^- \subset G_S$, we have
    $p',q' \in G_S$ and
    \begin{equation} \label{eq:d-lam-UB-dr-}
        d^{\lam}_r(u^{-1}(p'),u^{-1}(q'))
        \leq d_{G_S}(p',q') \le d_r^-(p',q').
    \end{equation}
    By the triangle inequality  we have 
    \begin{equation} \label{eq:p'-q'-LB}
    \|p' - q'\| \ge \|u(p) - u(q)\| - 2r \log^2 n 
    \ge e^{-\eta}d_M(p,q) - 2r \log^2n \ge e^{-2\eta} r \log^3 n
    \end{equation}
    where the last bound is for $n$ sufficiently large since we assumed $d_M(p,q) \geq r \log^3 n$.  
    Equation \eqref{eq:p'-q'-LB} ensures we may apply $A^-$, and so by $A^-$ along with \eqref{eq:B-application} and \eqref{eq:d-lam-UB-dr-}  
    to see 
    \begin{align}
        rd^\lam_r(p,q) &\le rd^\lam_r(u^{-1}(p'),u^{-1}(q')) + 4r \log^2 n \le rd^-_r(p',q') + 4r \log^2n \le e^\eta \mu(\lowerlambda)\|p' - q'\| + 4r \log^2 n\,.  \label{eq:dlamr-UB}
        \end{align}
    By the distortion bounds on $u$ along with \eqref{eq:upper-lower-TC-bounds} we have \begin{align}
        e^\eta \mu(\lowerlambda)\|p' - q'\|  \leq e^{2\eta + (\epsilon/2)} \mu(\lam) d_M(u^{-1}(p'),u^{-1}(q')) \leq e^{2\eta + \epsilon/2} \mu(\lam) (d_M(p,q) + 2r \log^2 n) \label{eq:patches-UB-time-constant-eq}
    \end{align}
    where in the last line we used \eqref{eq:B-application}.  Recalling that $d_M(p,q) \geq r \log^3 n$, combining \eqref{eq:dlamr-UB} and \eqref{eq:patches-UB-time-constant-eq} for $n$ large enough shows \eqref{eq:M-local-need} completing the proof. 
\end{proof}

\section{Global metric comparison bounds and proof of \texorpdfstring{\cref{thm:main}}{Theorem 2}} \label{sec:global-bounds}

The main goal of this section is to prove \cref{thm:main}.  The main work is to prove a two-point metric comparison (\cref{lem:two-point}), after which we will deduce Gromov-Hausdorff convergence in \cref{sec:proof-of-main}.

\subsection{A two-point asymptotic} \label{sec:two-point}

As usual for a point $p \in M$ and parameters $n,r$ we will write $\mathring{p}$ to be the $d_M$-nearest vertices of $\mathring{G}_M^n(r)$ to $p$.  The main purpose of this subsection is to show that two-point distances scale appropriately.

\begin{prop} \label{lem:two-point}
    Let $M$ be a connected compact $d$-dimensional Riemannian manifold of volume $1$ with $d \geq 2$.  For each fixed $\lambda_0 > \lambda_c$, $\beta \in (0,1), \eps > 0$ and $p \neq q \in M$ there are constants $c = c(M,\lambda_0,\beta,\eps,p,q) > 0$ and $n_0 = n_0(M,\lambda_0,\beta,\eps,p,q)$ so that {if $r \le n^{-\beta}$ and $\lambda \ge \lambda_0$} we have
    \[
      \Prob \left( 
      \frac{r d^\lam_r(\mathring{p},\mathring{q})}{\mu(\lam) d_M(p,q)} \in 
      [e^{-\epsilon},e^\epsilon]
      \right)
      \geq 1 - \exp(-c \log^2 n)
    \]
    for all $n \geq n_0$.
\end{prop}

The strategy for proving \cref{lem:two-point} will be to consider a $d_M$-geodesic from $p$ to $q$ and apply the estimate from \cref{lem:patches} at well-chosen points near the geodesic.  We will use compactness to show that we need only consider finitely many points.  With this in mind, using \cref{lem:patches}, for each $x \in M$,
    choose $\delta_x > 0$ such that the event
    \[
    \mathcal{G}_{x}
    :=
    \left\{ \begin{array}{c}
      \forall p', q' \in \mathring{G}^n_M(r) \cap B_M(x,\delta_x)
      \mbox{ such that } d_M(p',q') \ge r \log^3 n, \\
      r d^\lam_r(p',q')/[\mu(\lam) d_M(p',q')] \in 
      [e^{-\epsilon/2},e^{\epsilon/2}]
      \end{array} \right\}
    \]
    satisfies $\Prob(\mathcal{G}_x) \geq  1 - \exp(- c \log^2 n)$.  By compactness of $M$, there is a collection of $x_1,\ldots,x_N \in M$ for $N < \infty$ such that $M \subset \bigcup_{j = 1}^N B_M(x_j,\delta_{x_j})\,.$  We define the open cover $$\mathcal{U} = \left\{ B_M(x_j,\delta_{x_j})\right\}_{j = 1}^N$$
    and note that $\mathcal{U}$ depends only on $\eps, \lambda_0$ and $M$.
    
    We also recall the event $$\mathcal{B} = \{ \text{for all }y \in M : d_M(y,\mathring{y}) \leq r \log^2 n\}\,.$$
    By \cref{lem:all-near-giant} we have $\P(\mathcal{B}^c) \leq \exp(-c \log^2 n).$
    We start with an upper bound.
    
    \begin{lem}\label{lem:two-point-UB}
        In the context of \cref{lem:two-point} there are  $c > 0,n_0$ so that  $\Prob\left( 
      \frac{r d^\lam_r(\mathring{p},\mathring{q})}{\mu(\lam) d_M(p,q)} \leq 
      e^\epsilon
      \right) \geq 1 -  \exp(-c \log^2 n)$ for all $n \geq n_0$.
    \end{lem}
    \begin{proof}
        Consider a geodesic in $M$ from $p$ to $q$ and let 
        $p=y_0, y_1,...,y_\ell =q$
        be a sequence of (distinct) points along the geodesic
        such that each consecutive pair of points
        lies in a common element of $\mathcal{U}$
        and we have
        \begin{equation} \label{eq:UB-geodesic}
           d_M(p,q) = \sum_{i=1}^\ell d_M(y_{i-1},y_i).
        \end{equation}
        We claim that for sufficiently large $n$  we have
    \begin{equation} \label{eq:UB-containment}
       \bigcap_{j=1}^N \mathcal{G}_{x_j} \cap \mathcal{B}
       \subset
       \{
       r d_r^\lam(\mathring{p}, \mathring{q})
       \le e^{\epsilon} \mu(\lam) d_M(p,q)
       \}.
    \end{equation}
    To begin with, the triangle inequality shows \begin{equation} \label{eq:2pt-UB-1}
         d_r^\lam(\mathring{p}, \mathring{q})
        \le
        \sum_{i=1}^\ell
        d_r^\lam(\mathring{y}_{i-1}, \mathring{y}_i)\,.
    \end{equation}
    On the event $\mathcal{B}$  the triangle
    inequality
    gives the lower bound
    \[
       d_M(\mathring{y}_{i-1}, \mathring{y}_i)
       \ge d_M(y_{i-1}, y_i) - 2r \log^2 n \geq r \log^3 n ,
    \]
    for 
    all $i=1,...,\ell$ when $n$ is sufficiently large.
    Therefore
    the event $\bigcap_{j=1}^N \mathcal{G}_{x_j}$
    gives us 
    \begin{equation}\label{eq:2pt-UB-2}
    r d_r^\lam(\mathring{y}_{i-1}, \mathring{y}_i)
    \le e^{\epsilon/2} \mu(\lambda)d_M(\mathring{y}_{i-1}, \mathring{y}_i).
    \end{equation}
    On the event $\mathcal{B}$ we have \begin{equation}\label{eq:2pt-UB-3}
        d_M(\mathring{y}_{i-1}, \mathring{y}_i) \leq d_M(y_{i-1},y_i) + 2 r \log^2 n\leq e^{\eps/2} d_M(y_{i-1},y_i)
    \end{equation}
    where the last inequality holds for $n$ large enough as a function of $\eps, \beta$ and $d_M(y_{i-1},y_i)$.  Combining \eqref{eq:2pt-UB-1}, \eqref{eq:2pt-UB-2} and \eqref{eq:2pt-UB-3}  with \eqref{eq:UB-geodesic} shows \begin{align*}
        r d_r^\lam(\mathring{p}, \mathring{q}) \leq e^{\eps} \mu(\lambda) d_M(p,q)\,.
    \end{align*}
    This shows \eqref{eq:UB-containment}.  Combining with \cref{lem:patches} and \cref{lem:all-near-giant} completes the proof.
    \end{proof}

We now show the analogous lower bound.

\begin{lem}\label{lem:two-point-LB}
      In the context of \cref{lem:two-point} there are $c > 0,n_0$ so that $\Prob\left( 
      \frac{r d^\lam_r(\mathring{p},\mathring{q})}{\mu(\lam) d_M(p,q)} \geq 
      e^{-\epsilon}
      \right) \geq 1 - \exp(-c \log^2 n)$ for $n \geq n_0.$ 
\end{lem}
\begin{proof}
    We will show that for $n$ sufficiently large we have 
    \begin{equation}\label{eq:LB-containment}
        \bigcap_{j=1}^N \mathcal{G}_{x_j} \cap 
        \mathcal{B}
        \subset
        \{r d_r^\lam(\mathring{p}, \mathring{q})
        \ge e^{-\epsilon} \mu(\lam) d_M(p,q)\}.
    \end{equation}

    To begin with, we first claim that we may find well-separated points along a $d_r^\lambda$ geodesic from $\mathring{p}$ to $\mathring{q}$. 
    \begin{claim}\label{cl:find-geodesic}
         On the event $\mathcal{B}$, we may find points $\mathring{p} = v_0,v_1,...,v_K = \mathring{q}$ so that for all $i$ we have $v_{i-1}$ and $v_i$ are in the same element of $\mathcal{U}$, $ r \log^3 n \le d_M(v_{i-1}, v_i)$, and $ \sum_{k=1}^K d_r^\lam(v_{k-1}, v_k)
       = d_r^\lam(\mathring{p}, \mathring{q}).$
    \end{claim}
    \begin{proof}[Proof of \cref{cl:find-geodesic}]
        Let $\alpha$ be the Lebesgue covering number of the open cover $\mathcal{U}$, meaning that if $x,y \in M$ with $d_M(x,y) \leq \alpha$ then there is an element of $\mathcal{U}$ containing both $x$ and $y$.  For $n$ large enough, the event $\mathcal{B}$ ensures that $d_M(\mathring{p},\mathring{q}) \geq r \log^3 n$.  For $r$ so that $r \log^3n + r \leq \alpha /2$, we may find points $\mathring{p} =v_0,v_1,\ldots,v_K = \mathring{q}$ on a $d_r^\lambda$-geodesic from $\mathring{p}$ to $\mathring{q}$ so that $r \log^3 n \leq d_M(v_{i-1},v_i) \leq \alpha$.  By the definition of $\alpha$, this shows that for each $v_{i-1},v_i$ there is an element of $\mathcal{U}$ containing both.  By the definition of the geodesic, we also have $\sum_{k=1}^K d_r^\lam(v_{k-1}, v_k)
       = d_r^\lam(\mathring{p}, \mathring{q})$ completing the claim.
    \end{proof}
    Letting $v_0,\ldots,v_K$ as in \cref{cl:find-geodesic}, on the event  $\bigcap_{j=1}^N \mathcal{G}_{x_j} \cap 
       \mathcal{B}$ we may bound

    \begin{equation*}
        r d_r^\lam(v_{k-1},v_k) \geq e^{-\eps/2} \mu(\lambda) d_M(v_{k-1}, v_k)
    \end{equation*}
    and so  \begin{equation}\label{eq:LB-dr-geodesic}
        r d_r^\lam(\mathring{p}, \mathring{q}) \geq \sum_{k = 1}^K e^{-\eps/2} \mu(\lambda) d_M(v_{k-1},v_k) \geq e^{-\eps/2} \mu(\lambda) d_M(\mathring{p},\mathring{q})
    \end{equation}
    where the last inequality is the triangle inequality.  On the event $\mathcal{B}$ and $n$ sufficiently large we have \begin{equation}\label{eq:LB-apply-B}
         d_M(\mathring{p},\mathring{q}) \geq d_M(p,q) - 2 r \log^2 n  \geq e^{-\eps/2} d_M(p,q)\,.
    \end{equation}
    Combining \eqref{eq:LB-dr-geodesic} and \eqref{eq:LB-apply-B} shows that 
    $$rd_r^\lam(\mathring{p}, \mathring{q}) \geq e^{-\eps} \mu(\lambda) d_M(p,q)$$
    thus establishing \eqref{eq:LB-containment}.  Applying \cref{lem:patches} and \cref{lem:all-near-giant} completes the proof.    
\end{proof}

\begin{proof}[Proof of \cref{lem:two-point}]
    Combining \cref{lem:two-point-UB} and \cref{lem:two-point-LB} completes the proof.
\end{proof}
    
\subsection{Proof of \texorpdfstring{\cref{thm:main}}{Theorem 2}} \label{sec:proof-of-main}

We start with a tool for upper bounding the Gromov-Hausdorff distance between two metric spaces. Recall that an $\epsilon$-net in a metric space 
$(X,d_X)$ is a subset $X_\epsilon \subset X$
such that for any $x \in X$,
there exists some $x' \in X_\epsilon$
with $d_X(x,x') \le \epsilon$.
We use the following upper bound on Gromov-Hausdorff distance which appears, for instance, in \cite[Thm.~6.14(2)]{tuzhilin2020lectures} (see also Gromov's book \cite[Prop.~3.5b]{Gromov07}):

\begin{thm} \label{GH-convergence}
    Suppose that $(X,d_X)$ and $(Y,d_Y)$ are metric spaces
    and let $\epsilon > 0$.
    Suppose that $X_{\epsilon} \subset X$
    and $Y_{\epsilon} \subset Y$
    are $\epsilon$-nets and that there
    exists a surjective function $f:X_\epsilon \to Y_\epsilon$
    such that
    \[
        \sup_{x, x' \in X_\epsilon}
        |d_X(x,x') - d_Y(f(x),f(x'))| \le \epsilon.
    \]
    Then
    $d_{\mathrm{GH}}(X,Y) \le 2\epsilon$.
\end{thm}

We will prove \cref{thm:main} by applying \cref{GH-convergence}.  To begin with, fix $\eps > 0$ and let $X = X_\eps \subset (M,\mu(\lambda)d_M)$ be a finite $(\eps/5)$-net in the metric $\mu(\lambda_0) d_M$; note that $X$ depends only on $M,\eps$ and $\lambda_0$. Recall that by \cref{lem:limit-time-constant} we have that $X_\eps$ is an $(\eps/5)$-net in $\mu(\lambda)d_M$ for all $\lambda \geq \lambda_0$.  As usual, for a point
    $x \in M$ denote by $\mathring{x} = \mathring{x}(r)$
    the (random) vertex of $\mathring{G}^n_M(r)$
    which is $d_M$-nearest to $x$.  Define the function $f_r:X \to \mathring{G}_M^n(r)$ by $f_r(x) := \mathring{x}.$
    We will first use \cref{lem:two-point} to show that $f_r$ has small distortion: 
    \begin{lem}\label{lem:f-bounded-distortion}
         Let $M$ be a connected compact $d$-dimensional Riemannian manifold of volume $1$ with $d \geq 2$.  For each fixed $\lambda_0 > \lambda_c$, $\beta \in (0,1), \eps > 0$, let $X$ be a finite $(\eps/5)$-net of $M$.  Then there are constants $c = c(M,\lambda_0,\beta,\eps,X) > 0$ and $n_0 = n_0(M,\lambda_0,\beta,\eps,X)$ so that {if $r \le n^{-\beta}$ and $\lambda \ge \lambda_0$} we have
        $$\Prob\left(\sup_{x, x' \in X}
        |\mu(\lam)d_M(x,x') - r d^\lam_r(f_r(x),f_r(x'))| > \epsilon/2\right) \leq  \exp(-c \log^2 n) $$
        for $n \geq n_0$. 
    \end{lem}
    \begin{proof}
         Choose $\gamma > 0$ sufficiently small so that 
    $\max(1 - e^{-\gamma}, e^\gamma - 1) \mu(\lam_0) \mathrm{diam}(M)
    \le \epsilon/2$.  We then see that, for each $x,z \in X$
    we have
    \begin{align*}
        \left\{
        \frac{r d^\lam_r(\mathring{x},\mathring{z})}{\mu(\lam) d_M(x,z)} \in 
        [e^{-\gamma},e^\gamma]
        \right\} 
        \subset 
        \left\{
        |\mu(\lam) d_M(x,z) - r d_r^\lam(f_r(x), f_r(z))|
        \le \epsilon/2
        \right\}.
    \end{align*}
    Therefore we have that
    \begin{align*}
        \Prob\left(\sup_{x, x' \in X}
        |\mu(\lam)d_M(x,x') - r d^\lam_r(f_r(x),f_r(x'))| > \epsilon/2\right) 
        &\le
        \sum_{x, x' \in X}
        \Prob\left(
        |\mu(\lam)d_M(x,x') - r d^\lam_r(f_r(x),f_r(x'))| > \epsilon/2\right) \\
        &\le
        \sum_{x,x' \in X}
        \Prob\left(
        \frac{r d^\lam_r(\mathring{x},\mathring{x}')}{\mu(\lam) d_M(x,x')} \notin 
        [e^{-\gamma},e^\gamma]
        \right)\,.
    \end{align*}
    Applying \cref{lem:two-point} completes the proof.
    \end{proof}

    We next show that $f_r$ is typically an $(\eps/2)$-net in $(\mathring{G}_M^n(r),rd_r^\lambda).$

    \begin{lem}\label{lem:f-gives-eps-net}
        Let $M$ be a connected compact $d$-dimensional Riemannian manifold with $d \geq 2$.  For each fixed $\lambda_0 > \lambda_c$, $\beta \in (0,1), \eps > 0$, let $X$ be a finite $(\eps/5)$-net of $M$.  Then there are constants $c = c(M,\lambda_0,\beta,\eps,X) > 0$ and $n_0 = n_0(M,\lambda_0,\beta,\eps,X)$ so that {if $r \le n^{-\beta}$ and $\lambda \ge \lambda_0$} we have $$\P(f_r(X) \text{ is not an } (\eps/2)\text{-net in }(\mathring{G}_M^n(r),rd_r^\lam ) ) \leq  \exp(-c \log^2n )$$
        for all $n \geq n_0.$
    \end{lem}
    \begin{proof}
        Let $\eta > 0$ be chosen as in \cref{sec:setting-up-processes} and note we have $2 e^\eta (\eps/5) \leq \eps/2$. We may then use \cref{lem:patches} along with compactness of $M$ to find a refinement $X' \supset X$ so that for each $x \in X'$ we have $\delta_x > 0$ so that $\{B_M(x,\delta_x)\}_{x \in X'}$ covers $M$ and so that if we define the event
        \[
            \mathcal{G}'_{x} :=
            \left\{
                \begin{array}{c}
                \forall p, q \in \mathring{G}^n_M(r) \cap B_M(x,\delta_x)
                \mbox{ such that } d_M(p,q) \ge r\log^3 n, \\
                r d^\lam_r(p,q)/[\mu(\lam) d_M(p,q)] \in 
                [e^{-\eta},e^{\eta}]
                \end{array}
            \right\}
        \]
        then $\P(\mathcal{G}'_{x}) \geq 1 - \exp(-c \log^2 n)\,.$   For a pair $x,y \in X'$ we also define $$T_{x,y} = \left\{ \frac{r d_r^\lam(\mathring{x},\mathring{y})}{\mu(\lam) d_M(x,y)} \in [e^{-\eta}, e^\eta] \right\}\,.$$  By \cref{lem:two-point}, we have that $\P(T_{x,y}) \geq 1 -  \exp(-c \log^2 n)\,.$  Thus, using \cref{lem:all-near-giant}, it is sufficient to show the inclusion \begin{equation}
            \bigcap_{x \in X'}\mathcal{G}'_{x}\cap \mathcal{B} \cap \bigcap_{x,y \in X'} T_{x,y} \subset \{f_r(X) \text{ is a }(\eps/2)\text{-net in }(\mathring{G}_M^n(r), r d_r^\lam)\}\,.
        \end{equation}
        To see this, let $y \in \mathring{G}_M^n(r)$ be arbitrary.  By requiring all $\delta_x$ satisfy $\delta_x \leq \eps / (5 \mu(\lambda_0))$, we see that there is some $z \in X'$ so that $y \in B_M(z,\delta_z)$ and $\mu(\lam) d_M(z,y) \leq \eps/5$. 

        \begin{claim}\label{claim:net-dist-UB}
            On $\bigcap_{x \in X'}\mathcal{G}'_{x}\cap \mathcal{B} \cap \bigcap_{x,y \in X'} T_{x,y}$ we have $r d_r^\lam(\mathring{z},y) \leq e^\eta \mu(\lambda) \max\{d_M(\mathring{z},y),9 r \log^{3}n  \}.$
        \end{claim}
        \begin{proof}[Proof of \cref{claim:net-dist-UB}]
            To see this, first note that 
            if $d_M(\mathring{z},y) > r \log^3 n$,
            then by $\mathcal{G}'_{z}$ we have 
            $r d^\lam_r(\mathring{z}, y) 
            \le  e^{\eta} \mu(\lambda) d_M(\mathring{z}, y)$.
            If instead we have $d_M(\mathring{z},y) \le r \log^3 n$
            and there exists some $w \in \mathring{G}^n_M(r)$
            with $d_M(\mathring{z},w) \in [3r\log^3 n, 4r\log^3 n]$
            then we have $d_M(y,w) \ge 2r\log^3 n$.  Since $y,w \in \mathring{G}_M^n(r)$ with $2 r\log^3 n \leq d_M(y,w) \leq 5 r \log^3 n,$ for $n$ large enough we have that there is some $q \in X'$ with $y,w \in B_M(q,\delta_q)$ implying that
            $$r d_r^\lam(y,w) \leq e^\eta  \mu(\lambda) d_M(y,w) \leq 5 e^\eta  \mu(\lambda) r \log^3 n\,.$$ 
            Since we also have $rd_r^\lambda(\mathring{z},w) \leq e^{\eta} \mu(\lambda) d_M(\mathring{z},w) \leq 4 e^\eta \mu(\lambda) r \log^3 n,$ we then see that in this case we have $rd_r^\lam(\mathring{z},y) \leq 9 e^{\eta} \mu(\lambda) r \log^3 n$.
            Thus, the only remaining case is when $d_M(\mathring{z},y) \leq r \log^3 n$ and there is no point
            $w \in \mathring{G}^n_M(r)$ with
            $d_M(\mathring{z},w) \in [3r\log^3 n, 4r\log^3 n]$.  In this case, we have that $\mathring{G}^n_M(r)$ is contained in
            a $d_M$-ball of radius $3r \log^3 n$ about
            $\mathring{z}$.  However, for each pair of distinct points $x,y \in X'$, the event $\mathcal{B} \cap T_{x,y}$ shows that  $\mathring{G}^n_M(r)$ has $d_M$-diameter $\Omega(1)$, thus showing that this final case does not occur. 
            This completes the proof of the claim.
        \end{proof}

        To complete the proof of the lemma, note that there exists $x \in X$ so that $\mu(\lam) d_M(x,z) \leq \eps/5$ and that $rd_r^\lam(f_r(x),\mathring{z}) \leq e^\eta \mu(\lambda)  d_M(x,z)$ by the event $T_{x,z}$. This implies \begin{align*}
            rd_r^\lam(f_r(x),y) \leq rd_r^\lam(f_r(x),\mathring{z}(r)) + rd_r^\lam(\mathring{z},y) \leq e^{\eta}\left(\eps/5 + \eps/4 + 9 r\log^3 n \right) \leq \eps/2
        \end{align*}
        where the second inequality is by \cref{claim:net-dist-UB} along with the event $\mathcal{B}$ showing $\mu(\lambda) d_M(\mathring{z},y) \leq \eps/4$; the last inequality follows from taking $n$ sufficiently large.
    \end{proof}

\begin{proof}[Proof of \cref{thm:main}]
    Define $f_r$ as above and $X$ to be a finite $\eps/(5 \mu(\lambda_0))$-net in $(M,d_M)$ and note that $\mu(\lambda_0) \geq \mu(\lambda)$ by \cref{lem:limit-time-constant}.   \cref{lem:f-bounded-distortion} shows $f_r$ has distortion at most $\eps/2$ into $(\mathring{G}_M^n(r),rd_r^\lam)$ with probability at least $1 - \exp(-\Omega(\log^2n)).$  Defining $Y = f_r(X)$ we have that $Y$ is an $\eps/2$-net in $(\mathring{G}_M^n(r),rd_r^\lam)$ with probability at least $1 - \exp(-\Omega(\log^2 n))$ by \cref{lem:f-gives-eps-net}.  Applying \cref{GH-convergence} completes the proof.
\end{proof}

\section{Proof of \texorpdfstring{\cref{thm:sparse}}{Theorem 3}} \label{sec:denser}

Rather than considering average degree $\Delta \in [n^{1-\beta},\alpha n]$ for some $\alpha > 0$ and $\beta \in (0,1)$ we may consider $r \in [n^{-\beta'},\alpha']$
for $\beta' \in (0,1/d)$ and $\alpha' > 0$.  It will be more convenient for us to work with these bounds on $r$ directly rather than $\Delta$.  We first prove connectivity:

\begin{lem}\label{lem:connected}
	Let $M$ be a connected compact $d$-dimensional Riemannian manifold of volume $1$ and fix $\beta \in (0,1/d)$. There are constants $n_0 = n_0(M,\beta)$ and $c = c(M,\beta)$ so that for all $n \geq n_0$ and $r \geq n^{-\beta}$ $$\P(G_M^n(r) \text{ is connected}) \geq 1 - \exp(- c n^{1 - d\beta})\,.$$
\end{lem}
\begin{proof}
	We will use \cref{Palm theory}.  Set $\gamma = 1/10$.
	For two points $x,y \in M$, consider a sequence of points $x = x_0,x_1,\ldots,x_K = y$ along a $d_M$-geodesic from $x$ to $y$ so that $d_M(x_{j-1},x_j) \leq (1 - 2\gamma) r$.   Further, we may choose this sequence of points so that $K \leq (1 - 2\gamma)^{-1} r^{-1} \mathrm{diam}(M) + 1 \leq O_M(n^\beta)$.  
	Note that for $x,y \in G_M^n(r)$ we have \begin{align*}
		\{x  \xleftrightarrow{G_M^n(r)} y\} \supset \bigcap_{j = 1}^{K-1} \{G_M^n(r) \cap B_M(x_j, \gamma r) \neq \emptyset \}\,.
	\end{align*}
	Note that for each $z \in M$ we have $$\P(G_M^n(r) \cap B_M(x_j, \gamma r) = \emptyset) = \exp(- n \mathrm{Vol}_M(B_M(x_j, \gamma r))) \leq \exp( - \Omega_M( n \gamma^d r^d)) = \exp(-\Omega_M(n^{1 - d\beta}))\,.$$
	\cref{Palm theory} then shows \begin{equation*}
		\P(G_M^n(r) \text{ is not connected}) \leq O_M( n^\beta \exp(- \Omega(n^{1 - d\beta}))) \leq \exp(- \Omega_M(n^{1 - d\beta}))\,. \qedhere 
	\end{equation*}
\end{proof}

We first prove a two-point estimate for the distances.
\begin{prop}\label{prop:two-point-sparse}
	Let $M$ be a connected compact $d$-dimensional Riemannian manifold of volume $1$ with $d \geq 2$.  For each fixed $\eps > 0$ and $\beta \in (0,1/d)$ there are $\alpha =  \alpha(M,\beta,\eps) > 0, c = c(M,\beta,\eps) > 0$  and $n_0 = n_0(M,\beta,\eps)$ so that the following holds.  For $n \geq n_0$ and $r \in [n^{-\beta},\alpha]$, let $G_M^n(r)$ be the  Poisson random geometric graph on $M$.  Then for all $p,q \in M$ with $d_M(p,q) \geq \eps/100$ we have 
	\begin{equation*}
		\Prob \left( 
		\frac{r d^\lam_r(\mathring{p},\mathring{q})}{ d_M(p,q)} \in 
		[e^{-\epsilon},e^\epsilon]
		\right)
		\geq 1 - \exp(-c n^{1-d\beta})\,.
	\end{equation*}
\end{prop}

\begin{proof}
	For the lower bound, first note that $G = G_M^n(r)$ is connected with probability $1 - \exp(-cn^{1-d\beta})$ by \cref{lem:connected}.  For $\gamma > 0$ to be chosen later note that \begin{align*}
		\{G \text{ is connected}\} &\cap \{B_M(p, \gamma\eps) \cap G \neq \emptyset\} \cap \{B_M(q, \gamma\eps) \cap G \neq \emptyset\} \\
		&\subset \{d_r^\lambda(\mathring{p},\mathring{q}) \geq (d_M(p,q) - 2\gamma \eps)/r + 1\}
	\end{align*}
	For $\gamma$ small enough as a function of $\eps$ and taking $r \leq \alpha$ for $\alpha$ small as a function of $\eps > 0$ we see $$\{\infty > d_r^\lambda(\mathring{p},\mathring{q}) \geq (d_M(p,q) - 2\gamma \eps)/r + 1\} \subset \{rd_r^\lambda(\mathring{p},\mathring{q}) \geq e^{-\eps} d_M(p,q)\}\,.$$  Noting that $\P(B_M(z, \gamma\eps) \cap G \neq \emptyset) \geq 1- \exp(-\Omega_{M,\eps}(n))$ uniformly for $z \in M$ shows 
    $$\P(rd_r^\lambda(\mathring{p},\mathring{q}) \geq e^{-\eps} d_M(p,q)) \geq 1 - \exp(-\Omega_{M,\eps}(n^{1- d\beta}))\,.$$
	
	For the upper bound, fix $\eps > 0$ and let $\gamma > 0$ to be chosen later as a function of $\eps > 0$. 
	For points $p,q \in G$ with $d_M(p,q) \geq \eps/100$, choose points $p = x_0,x_1,\ldots,x_K = q$ along a $d_M$-geodesic so that  \begin{equation}\label{eq:K-def}
		d_M(x_{j-1},x_j) \leq (1 - 4\gamma)r,\quad  \left|K -  \left((1 - 4\gamma) r\right)^{-1} d_M(p,q) \right| \leq 1\,.  \end{equation}
    We now note that \begin{align}\label{eq:d-LB-sparse}
		\P(  d_r^\lambda(p,q) \leq K ) \geq  \P\left(\bigcap_{j=0}^K \{G_M^n(r) \cap B_M(x_j,\gamma r) \neq \emptyset \}\right)\geq 1- \exp(-\Omega_M(n^{1-d\beta}))\,.
	\end{align}
	For $r,\gamma$ small enough as a function of $\eps$ and $d_M(p,q)$ we have that $K \leq r^{-1} e^\eps d_M(p,q)$, completing the proof.
\end{proof}

We are now ready to prove convergence of metric spaces.

\begin{proof}[Proof of \cref{thm:sparse}]
	For fixed $\eps > 0$, let $X = X_\eps \subset M$ be a finite $\eps$-net in the metric $d_M$.  Note that we may assume that for all $x,x' \in X$ with $x \neq x'$ we have $d_M(x,x') \geq \eps/2$.  Define $f_r(x) = \mathring{x}$.  By \cref{prop:two-point-sparse} we have \begin{align*}
		\P\left(\sup_{x,x' \in X} |d_M(x,x') - rd_r^\lambda(f_r(x),f_r(x')) | \geq \eps \right)\leq C_{M,\eps} \exp\left( - c_{M,\eps} n^{1 - d\beta} \right)
	\end{align*}
for $r \leq \alpha = \alpha(M,\eps,\beta)$ and $n \geq n_0 = n_0(M,\eps,\beta)$. For any $x \in X$ and $z \in B_M(x,\eps) \cap G_M^n(r)$, the proof of \cref{prop:two-point-sparse} (in particular equations \eqref{eq:K-def} and \eqref{eq:d-LB-sparse}) shows that $$\P(rd_r^\lambda(\mathring{x},z) \geq 2\eps) \leq C_{M,\eps} \exp\left( - c_{M,\eps} n^{1 - d\beta} \right)\,.$$
By \cref{Palm theory}, this implies that $$\P(rd_r^\lambda(\mathring{x},z) \leq 2\eps \text{ for all }z \in B_M(x,\eps)) \geq  1 - \exp\left( - c_{M,\eps}' n^{1 - d\beta} \right)$$
thus showing that $f_r(X)$ is a $2\eps$ net with probability $1 - \exp(-\Omega(n^{1-d\beta}))$. \cref{GH-convergence} completes the proof.
\end{proof}

\section{Algorithmic implications} \label{sec:algorithmic-proofs}

We first show a basic geometric statement which will be the main step for deducing \cref{thm:algorithmic} from \cref{thm:degree}:

\begin{lem}\label{lem:GH-positive-implies-finite}
	Let $(M_1,d_1)$ and $(M_2,d_2)$ be compact metric spaces that are not isometric to each other. Then after possibly switching the roles of $M_1$ and $M_2$, there is some finite collection of points $x_1,\ldots,x_K \in M_1$ so that  there is no collection of points $y_1,\ldots,y_K \in M_2$ with $$d_1(x_i,x_j) = d_2(y_i,y_j) \text{ for all } 1 \leq i \leq j \leq K\,.$$
\end{lem}
\begin{proof}
    We prove the contrapositive.  Since $M_1$ is compact, we may find a countable dense set $\{x_j\}_{j \geq 1} \subset M_1$.  For each $n$, we may find a sequence $\{y_j^{(n)}\}_{j = 1}^n$ so that $$d_2(y_i^{(n)},y_j^{(n)}) = d_1(x_i,x_j) \quad \text{ for all }i < j \leq n\,.$$
    By compactness of $M_2$, we may find a subsequence of $\{n_a\}$ so that for each fixed $i$ we have that $y_i^{(n_a)} \to y_i$ for some $y_i \in M_2$.  By continuity of $d_2$ we have that  \begin{equation}\label{eq:embedding}
        d_2(y_i,y_j) = d_1(x_i,x_j) \qquad \text{ for all } i < j < \infty
    \end{equation}  We define an embedding $f:M_1 \to M_2$ by first defining $f(x_i) = y_i$.  For an arbitrary point $x \in M_1$, we may find a subsequence $n_j$ so that $x_{n_j} \xrightarrow{j \to \infty} x$, and in particular $\{x_{n_j}\}_j$ is a Cauchy sequence. By \eqref{eq:embedding} this implies $\{y_{n_j}\}_j$ is a Cauchy sequence; since $M_2$ is a compact metric space, this implies that $y_{n_j} \xrightarrow{j \to \infty} y \in M_2$.  We define\footnote{To see that $f$ is well-defined, note that if we have two subsequences $x_{n_j}$ and $x_{m_j}$ with $x_{n_j} \to x$ and $x_{m_j} \to x$ then we have $d_2(y_{n_i},y_{m_j}) = d_1(x_{n_i}, x_{m_j}) \leq d_1(x_{n_i},x) + d_1(x_{m_j},x)$ which tends to zero as each of $i,j \to \infty$.  The sequences $(y_{n_i})_i, (y_{m_j})_j$ thus have the same limit.} $f(x) = y$.  For all $x,z \in M_1$, we have that there are sequences $x_{n_i} \to x$ and $x_{m_i} \to z$ and so $$d_2(f(x),f(z)) = \lim_i d_2( f(x_{n_i}), f(x_{m_i})) = \lim_i d_1(x_{n_i}, x_{m_i}) = d_1(x,z)$$
    implying that $f:M_1 \to M_2$ is an isometric embedding.  Repeating the argument with the roles reversed gives us an isometric embedding $g:M_2 \to M_1.$  This implies that $g \circ f:M_1 \to M_1$ is an isometry.  Since $M_1$ is a compact metric space, this implies that $g \circ f$ is surjective (see e.g.\ \cite[Thm.~1.6.14]{burago2001course}). We claim that $f$ is also surjective, so take $y \in M_2$. Set $x = g(y) \in M_1$, and by surjectivity of $g \circ f$ there is some $z \in M_1$ so that $g(f(z)) = x$.  But since $g$ is an isometry, we have that $g$ is injective implying $y = f(z)$.  This shows that $f:M_1 \to M_2$ is a bijective isometry, completing the proof.
\end{proof}

From this we deduce a statistic that will separate our two manifolds with high probability.

\begin{cor}\label{cor:identify-finite}
	Let $M_1, M_2$ be compact Riemannian manifolds that are not isometric, with metrics $d_1$ and $d_2$. 
    Then after possibly switching the roles of $M_1$ and $M_2$, there is some finite collection of points $x_1,\ldots,x_K \in M_1$ along with a parameter $\eps > 0$ so that for all  $y_1,\ldots,y_K \in M_2$ we have $$\sum_{i,j} \left|d_1(x_i,x_j) - d_2(y_i,y_j)\right| \geq \eps\,.$$
\end{cor}
\begin{proof}
	Apply \cref{lem:GH-positive-implies-finite} to obtain points $x_1,\ldots,x_K \in M_1$.  Consider the function $F:M_2^K \to \R_{\geq 0}$ defined by $$F(y_1,\ldots,y_K) = \sum_{i,j} \left|d_1(x_i,x_j) - d_2(y_i,y_j)\right|\,.$$
	Note that $F$ is continuous on the compact set $M_2^K$ and never equal to $0$, implying it achieves some minimum $\eps > 0$.
\end{proof}

Conversely, it is not hard to see that if the Gromov-Hausdorff distance is small then one can always approximate pairwise distances: 

\begin{lem}\label{lem:GH-finite-errors}
    Suppose that $(M_1,d_1)$ and $(M_2,d_2)$ are metric spaces so that $d_{\GH}((M_1,d_1), (M_2,d_2)) \leq \eps$.  Then for all $x_1,\ldots,x_m \in M_1$ there is $y_1,\ldots,y_m \in M_2$ so that for all $i < j \leq m$ we have $$\left|d_1(x_i,x_j) - d_2(y_i,y_j)\right| \leq 4\eps\,.$$
\end{lem}
\begin{proof}
    By the definition of $d_{\GH}$ there is some metric space $(N,d_N)$ along with isometric embeddings $f:M_1 \to N$ and $g:M_2 \to N$ so that $d_N(f(M_1),g(M_2)) \leq 2\eps\,.$  This implies that for each $x_i$ pick some $y_i$ so that $d_N(f(x_i),g(y_i)) \leq 2\eps$.  By the triangle inequality we have \begin{align*}
        \left|d_1(x_i,x_j) - d_2(y_i,y_j) \right| &= |d_N(f(x_i),f(x_j)) - d_N(g(y_i),g(y_j)) | \\
        &\leq d_N(f(x_i),g(y_i)) + d_N(f(x_j),g(y_j)) \\
        &\leq 4\eps\,. \qedhere
    \end{align*}
\end{proof}

We handle the case of \cref{thm:algorithmic} when they have different dimensions first.  Here, the corresponding values of $r_j$ will radically differ and so the graph diameters will be a distinguishing statistic.

\begin{prop}\label{prop:diff-dimensions}
    Let $M_1$ and $M_2$ be connected compact Riemannian manifolds of dimension $d_1 > d_2 \geq 2$ and volume $1$.  Let $\Delta_c(d_j)$ denote the critical average degree for having a giant component. 
    There is $\alpha = \alpha(M_1,M_2)$ so that for each $\Delta_0 > \max\{\Delta_c(d_1),\Delta_c(d_2)\}$ there is $c = c(M_1,M_2,\Delta_0)$ so for all $\Delta \in [\Delta_0, \alpha n]$, given the adjacency matrices of the random geometric graphs  $G_{M_1}(n;r_1)$ and $G_{M_2}(n;r_2)$ of expected average degree $\Delta$, there is a $O(n^2 \Delta)$ algorithm which can correctly identify which graph came from $M_1$ and $M_2$ with probability at least $1- \exp(-c \log^2 n)\,.$ 
\end{prop}
\begin{proof}
    For $\alpha$ small enough, we have that $r_j = (\frac{\Delta}{n \alpha_{d_j}})^{1/d_j} (1 + o_{\alpha \to 0}(1))$ for $j \in \{1,2\}$ where $\alpha_d$ is the volume of the unit ball in $\R^d$.  In particular, for each $L > 0$ we may take $\alpha$ sufficiently small so that $r_1 / r_2 \geq L$.  For each fixed $\eps > 0$, we may apply \cref{thm:degree} to see that with probability at least $1 - \exp(-\Omega(\log^2n))$ we have 
    \begin{equation} \label{eq:diameters-bounds} 
    \mathrm{diam}(\mathring{G}_{M_1}(n;r_1)) \leq (1 + \eps)\mubar_{d_1}(\Delta) \mathrm{diam}(M_1) r_1^{-1} \  \text{ and } \  \mathrm{diam}(\mathring{G}_{M_2}(n;r_2)) \geq (1 - \eps)\mubar_{d_2}(\Delta) \mathrm{diam}(M_2) r_2^{-1}
    \end{equation}
    where we write $\mathrm{diam}(\mathring{G}_{M_1}(n;r_1))$ for the graph diameter and $\mathrm{diam}(M_1)$ for the metric diameter.  In particular, for $\alpha$ sufficiently small we have $$\frac{\mathrm{diam}(\mathring{G}_{M_1}(n;r_1))}{\mathrm{diam}(\mathring{G}_{M_2}(n;r_2))} \leq \frac{1}{2}$$
    provided we have \eqref{eq:diameters-bounds}.  We can find the graph diameter of each graph by running a breadth first search starting from each vertex; this identifies the giant component by \cref{cor:graph-diam-giant}.  Since a breadth first search has run-time $O(n\Delta)$, the algorithm runs in time $O(n^2 \Delta)$.  Choosing the graph of larger graph diameter provides the algorithm which succeeds with probability at least $1 - \exp(-c \log^2 n).$
\end{proof}

We need a basic fact about concentration of the number of edges. While obtaining an optimal bound appears to be challenging, the work of Bonnet and Gusakova \cite{bonnet2024concentration} shows that the number of edges in a random geometric graph in this level of generality has the appropriate number of edges (see Section 6.1 of \cite{bonnet2024concentration} for an argument showing that edge counts satisfy the assumptions of their Theorem 1.1):

\begin{lem}\label{lem:edge-concentration}
    Let $M$ be a connected compact $d$-dimensional Riemannian manifold of volume $1$.  For each $n$ let $G_M^n(r)$ be the Poisson random geometric graph with expected average degree $\Delta \in [1,n/2]$.  Then for each $\eps > 0$ there is $c = c(\eps,M) > 0$ so that if $e(G_M^n(r))$ is the number of edges in $G_M^n(r)$ then $$\P\left(\left| e(G_M^n(r)) - n\Delta/2 \right| > \eps n\Delta \right) \leq \exp\left(-c \log^2 n \right)\,.$$
\end{lem}

\begin{proof}[Proof of \cref{thm:algorithmic}]
    First note that if $d_1 \neq d_2$ then \cref{prop:diff-dimensions} applies and so we may assume $d_1 = d_2 = d$. 

    Apply \cref{cor:identify-finite} to find $\eps_1 = \eps_1(M_1,M_2) \in (0,1/2)$ and $K = K(M_1,M_2)< \infty$ so that (up to possibly swapping $M_1$ and $M_2$) there is $x_1,\ldots,x_K \in M_1$ so that for all $y_1,\ldots,y_K \in M_2$ we have \begin{equation} \label{eq:separation criterion}
    \sum_{i,j} \left|d_{M_1}(x_i,x_j) - d_{M_2}(y_i,y_j) \right| \geq \eps_1\,.
    \end{equation}
    We will apply \cref{thm:degree} for $\eps = \eps_2 := \eps_1 / (16 K^2)$.  For $i \in \{1,2\}$ define the event $\mathcal{G}_i$ via \begin{align*}
        \mathcal{G}_i = \left\{d_{\GH}\left((\mathring{G}_{M_i}(n;r_i),  \frac{r_i}{\mubar(\Delta)} d_{G_i}),(M_i,  d_{M_i}) \right) \leq \eps_2\right\}\,.
    \end{align*} 
    We then note that by \cref{lem:GH-finite-errors}, this implies that on event $\mathcal{G}_1$, we may find points $z_1,\ldots,z_K \in \mathring{G}_1$ so that \begin{equation}\label{eq:M-1-zs}
        \sum_{i,j} |d_{M_1}(x_i,x_j) - \frac{r_1}{\mubar(\Delta)}d_{G_1}(z_i,z_j)| \leq K^2 4 \eps_2 \leq  \frac{\eps_1}{4}\,.
    \end{equation}
    Conversely, on event $\mathcal{G}_2$, for all $w_1,\ldots, w_K \in \mathring{G}_2$, there are points $y_1,\ldots,y_K$ so that $$\sum_{i,j} |d_{M_2}(y_i,y_j) - \frac{r_2}{\mubar(\Delta)}d_{G_2}(w_i,w_j)| \leq \frac{\eps_1}{4}\,.$$
     Letting $\alpha_d$ denote the volume of the unit ball in $\R^d$, we note that for $\alpha$ sufficiently small we have $$r_k = (1 + o_{\alpha \to 0}(1))\left(\frac{\Delta}{n \alpha_d} \right)^{1/d}$$  
    for each $k \in \{1,2\}$.  Further, if we apply \cref{lem:edge-concentration} and define $r = \left(\frac{2e(G_{M_1}(n;r_1))}{n^2 \alpha_d}\right)^{1/d}$ then note that with probability at least $1 - \exp(-\Omega(\log^2 n))$ we have $$e^{-\eps_1/100} \leq \frac{r_k}{r} \leq e^{\eps_1/100}$$
    for $k \in \{1,2\}$ and $\alpha$ sufficiently small.  We may then use \eqref{eq:separation criterion} to see that  for all $w_1,\ldots,w_K \in \mathring{G}_2$ we have $$\sum_{i,j} \left|d_{M_1}(x_i,x_j) - \frac{r}{\mubar(\Delta)}d_{G_2}(w_i,w_j) \right| \geq \frac{3\eps_1}{4} $$
    and for $z_1,\ldots,z_K \in G_1$ satisfying \eqref{eq:M-1-zs} we have 
$$\sum_{i,j} \left|d_{M_1}(x_i,x_j) - \frac{r}{\mubar(\Delta)}d_{G_1}(z_i,z_j) \right| \leq \frac{\eps_1}{2}\,.$$
    Our algorithm is defined as follows.  First use a breadth-first search started at each vertex to find all pairwise distances in the metrics $d_{G_i}$ in run-time $O(n^3)$; this also allows us to find the largest component by \cref{cor:graph-diam-giant}
    simultaneously.  We may also compute $r = \left(\frac{2e(G_{M_1}(n;r_1))}{n^2 \alpha_d}\right)^{1/d}$  and $\Delta' = 2e(G_{M_1}(n;r_1))/n$ in this time as well.  For $k \in \{1,2\}$ compute  $$a_k =\min_{w_1,\dots,w_K \in \mathring{G}_k} \sum_{i,j} \left|d_{M_1}(x_i,x_j) - \frac{r}{\mubar(\Delta')}d_{G_k}(w_i,w_j) \right|$$
    in time $O(n^K)$ by checking all $K$-tuples.  If $a_1 \leq a_2$ output that $G_1$ came from $M_1$, otherwise output that $G_2$ came from $M_1$.
    
    Since $\Delta' = (1 + o_{\alpha \to 0}(1))\Delta$ and $\mubar$ is continuous by \cref{Continuity of Time Constant}, we see that for $\alpha$ small enough we have that $a_1 \leq 9\eps_1/16$ and $a_2 \geq 11\eps_1/16$ and so this algorithm outputs the correct outcome with probability at least $ 1 - \exp(-c \log^2 n)\,.$
\end{proof}

\section*{Acknowledgments}
K.D.\ and M.M.\ are partially supported by NSF grant DMS-2246624. M.M.\ is additionally supported by NSF grant DMS-2336788.  The authors thank Elchanan Mossel for comments on the introduction of a previous draft.

\bibliographystyle{abbrv}
\bibliography{refs}

\appendix
\section{Some results for geometric random graphs on \texorpdfstring{$\R^d$}{Rd} with
distorted metric and intensity measure} \label{sec:appendix}

\begin{proof}[Proof of  \cref{lem:non-uniform-uniqueness}]
    We will adapt the proof of its homogeneous analogue, \cite[Prop.~10.13(ii)]{Pen03}.
    The key observation is that, since
    $\sigma(\cdot,\cdot) \le \|\cdot-\cdot\|$ and $\frac{d\nu}{d\mathcal{L}} \ge \lambda_0$, the graph $G_s$ is contains its homogeneous counterpart
    as a subgraph, allowing us to use monotonicity
    of certain events. The metric $\sigma$ is used
    to determine the edges of $G_s$, and should
    be used to define balls in the ``crossing events''
    in \cite[Sec.~10.2]{Pen03}.
    The fact that $\sigma$ and $\nu$ are inhomogeneous means that
    technically all the lemmas have to be stated and proven
    for $x + B(s)$ rather than $B(s)$, with all bounds uniform
    over $x$, but this does not change the proofs in an essential way.

    Note first that the existence of a component of diameter at least $\phi_s$ follows from applying \cite[Prop.~10.13]{Pen03} at $\lambda_0$, and so we will focus on the non-existence of other components.
    We now  
    show how to adapt \cite[Prop.~10.13(ii)]{Pen03}.  The proof of \cite[Prop.~10.13(ii)]{Pen03} is broken into two cases, of $d \geq 3$ and $d = 2$. We describe these separately.
    
    For $d \geq 3$, we first give a quick overview of the logical chain of implications starting from \cite[Prop.~10.13(ii)]{Pen03} and working backwards, and then describe the bounds attained.  Throughout, we will adhere as closely as possible to the notation in \cite{Pen03} so one may directly translate.
    The claim \cite[Prop.~10.13(ii)]{Pen03} follows from \cite[Lemma~10.12]{Pen03}, which shows that it is rare to have two components $\kappa$ and $\kappa'$ so that $\kappa$ crosses in the $k$ coordinate and that the (Euclidean) diameter of the projection of $\kappa'$ to the $k$ coordinate is at least $\phi_s$. This is performed by considering a slab decomposition and using \cite[Lemma~10.10]{Pen03} to bound the probability of a crossing event on each slab. 
    This is where the lower bound on $\sigma$ is used:
    it guarantees that (for a sufficiently large
    choice of slab width $K'$) any component which has vertices
    in two slabs must have vertices in all 
    the intermediate slabs; this property is necessary to establish the inclusion
    \[
        G'_{x,y,s} \subset \bigcap_{j-1}^{[(\phi_s - 1)/K']} A'_j
    \]
    where the above events are defined as in \cite[Lemma~10.12]{Pen03}.
    
    Here, in the statement of \cite[Lemma~10.10]{Pen03}, in
    order to account for the inhomogeneity of the metric, the 
    lower bound we want is
    \[
        \liminf_{s \to \infty}
        \inf_{x \in \R^d}
        \inf_{A, B \subset N_x}
        \frac{\Prob\left[C(A; \mathcal{H}_\nu \cap x + S(K',s))
        \cap C(B; \mathcal{H}_\nu \cap x + S(K',s)) \ne 
        \emptyset \right] }
        {\Prob \left[ C(A; \mathcal{H}_\nu \cap x + S(K',s)) \ne 
        \emptyset;
        C(B; \mathcal{H}_\nu \cap x + S(K',s) \ne \emptyset
        \right]}
        > 0
    \]
    where $\mathcal{H}_\nu$ is the point configuration of
    the Poisson process with intensity $\nu$,
    which we elsewhere refer to as $X_\nu$; we take
    \[
    N_x := \{ y \in x +(-2,0)\times[0,s]^{d-1} :
    \sigma(y, x + S(K',s)) < 1\};
    \]
    and $S(K',s)$ and $C(\cdot ; \cdot)$
    are defined as in \cite[Lemma~10.10]{Pen03}.    
    
    The only percolation input into \cite[Lemma~10.10]{Pen03} is \cite[Lemma~10.8]{Pen03} which is in fact a monotone event, and so it directly applies to our inhomogeneous setting.
    Note also that in the case $\lam = \mu$ (borrowing the notation from \cite[Lemma~10.10]{Pen03}), the lower
    bound obtained improves as $\lam \to \infty$, and
    so we get a constant depending only on $\lam_0 > \lam_c$
    which holds for all $\frac{d\nu}{d\mathcal{L}} \ge \lam_0$.
    
    Putting the pieces together, we note that the proof of \cite[Lemma~10.12]{Pen03} yields an upper bound of $$\Prob\left(
        \begin{array}{c}
            \mbox{the second largest component of } G_s \\
            \mbox{ has Euclidean diameter at least } \phi_s
        \end{array}
        \right) \leq C_{d} s^d \exp(-c \phi_s)$$
        for $c = c(d,\lambda_0) > 0$.

    For the case of $d = 2$, the proof of \cite[Prop.~10.13(ii)]{Pen03} is more explicitly monotone: one covers the square with $O_d(s^2)$ many vertical and horizontal dominoes of aspect ratio $2$ and side-length $c_{d} \phi_s$, applies \cite[Lemma~10.5]{Pen03} to assert that each horizontal domino has a horizontal crossing and each vertical domino has a vertical crossing with probability at least $1 - \exp(-c \phi_s)$, and then notes that any component of large metric diameter must intersect the component constructed in this manner.  Since the only percolation input to this argument is the monotone event guaranteeing a crossing for each domino, it directly implies the $d=2$ case of \cref{lem:non-uniform-uniqueness}. 
\end{proof}

\begin{proof}[Proof of \cref{non-uniform-Antal-Pisztora}]
    We indicate how to extract this statement from the proof of its homogeneous analogue \cite[Lemma 3.4]{YCG11}.  The proof of \cite[Lemma 3.4]{YCG11} is a renormalization technique; we introduce some of their notation here.  For a parameter $M$ which will be fixed but large they define $B(M) = [-M/2,M/2]^d$ and $B(5M/4)$ analogously.  They then define their boxes $B_z = B(M) + \{Mz\}$ and $B_z^+ = B(5M/4) + \{Mz\}$.  They define events $A_z$ depending only on the boxes $B_z$ and $B_z^+$.  The only properties we need about these events are that: (1) $A_z$ and $A_w$ are independent if $\|z - w\|_\infty \geq 2$;  (2) for each fixed $\delta > 0$ one may make $M$ large enough so that $\P(A_z) \geq 1 - \delta$ for each $z$, uniformly in $z$.  Property (1) holds for us as well (provided $M$ is large enough).  Property (2) is proven in \cite{YCG11} by the homogeneous analogue of Lemma \ref{lem:non-uniform-uniqueness} and so Lemma \ref{lem:non-uniform-uniqueness} proves property (2) for us as well.

    As in \cite{YCG11}, by $O(1)$-dependence, for any $p_1$, one may find $M_0$ so that for all $M \geq M_0$ the Bernoulli field $\{\one\{A_z\} \}_{z \in \Z^d}$ stochastically dominates an i.i.d.\ Bernoulli-$p_1$ random field.  
    
    Following \cite{AP96}, they will alter a path witnessing $x \stackrel{G}{\longleftrightarrow} y$ in order to upper bound it in the macroscopic scale.  To this end, in the macroscopic scale $z \in \Z^d$ they call a box white if $A_z$ holds and black otherwise.  If $x \stackrel{G}{\longleftrightarrow}y$ then they let $a,b \in \Z^d$ be so that $x \in B_a$ and $y \in B_b$ and choose a path $a_0,a_1,\ldots,a_n \in \Z^d$ so that $a = a_0, a_n = b$ and $\|a_i - a_{i+1}\|_1 = 1$.  For a point $z \in \Z^d$ they define the cluster $C(z)$ to consist of $\{z\}$ if $z$ is white and the black connected component of $z$ together with its boundary if it is black.  Recall that the black vertices are stochastically dominated by i.i.d.\ $1 - p_1$ percolation.  They now define $$W = \bigcup_{j = 0}^n \bigcup_{z \in C(a_j)} B_z^+\,.$$

    They follow \cite{AP96} to alter the path witnessing $x \stackrel{G}{\longleftrightarrow} y$ to lie within $W$, which in our setting follows from translating \cite{AP96} verbatim (as \cite{YCG11} does).  The next step of \cite{YCG11} shows that if one has a path $x_1,\ldots,x_t$ of connected points in $G$ that is minimal that lies in $B(M)$ then one can upper bound $t$ in terms of $\mathcal{L}(B(M+1))$; this follows from noting---as in our case---that by minimality of the path, the balls of radius $1/2$ centered at $x_j$ are disjoint and lie in $B(M+1)$.  This guarantees that $d_G(x,y) \leq C' |W|$ for some $C'$ depending on $M$ and the dimension $d$.  They complete the proof by noting that if one takes $p_1$ large enough, one may bound $ |W|$ by $O_M(\|x - y\|_2)$ as in \cite{AP96}; this only depends on having $1 - p_1$ sufficiently small as a function of $d$ and so the proof is complete.
\end{proof}

\end{document}